\UseRawInputEncoding
\documentclass[12pt, reqno]{amsart}
\usepackage[margin=1in]{geometry}
\usepackage{amssymb,latexsym,amsmath,amscd,amsfonts}
\usepackage{latexsym}
\usepackage[mathscr]{eucal}
\usepackage{bm}
\usepackage{mathptmx}
\usepackage{amssymb}
\usepackage{amsthm}
\usepackage{dcolumn}
\usepackage[all]{xy}
\usepackage{enumitem}
\usepackage[utf8]{inputenc}

\def \qed {\hfill \vrule height6pt width 6pt depth 0pt}
\def\textmatrix#1&#2\\#3&#4\\{\bigl({#1 \atop #3}\ {#2 \atop #4}\bigr)}
\def\dispmatrix#1&#2\\#3&#4\\{\left({#1 \atop #3}\ {#2 \atop #4}\right)}
\newcommand{\beg}{\begin{equation}}
	\newcommand{\eeg}{\end{equation}}
\newcommand{\ben}{\begin{eqnarray*}}
	\newcommand{\een}{\end{eqnarray*}}

\newcommand{\C}{\mathbb C}

\newcommand{\N}{\mathbb N}
\newcommand{\T}{\mathbb T}

\newcommand{\E}{\mathbb E}

\newcommand{\Pe}{\mathbb P}
\newcommand{\D}{\mathbb D}
\newcommand{\G}{\mathbb G}

\newcommand{\lm}{\lambda}

\newcommand{\HS}{\mathcal{H}}
\newcommand{\Hex}{\mathbb{H}}
\newcommand{\al}{\alpha}
\newcommand{\DC}{\overline{\mathbb{D}}}
\newcommand{\B}{\mathbb{B}}
\newcommand{\LS}{\mathcal{L}}
\newcommand{\KS}{\mathcal{K}}
\newcommand{\la}{\langle}
\newcommand{\ra}{\rangle}

\newtheorem{thm}{Theorem}[section]
\newtheorem{cor}[thm]{Corollary}
\newtheorem{lem}[thm]{Lemma}

\newtheorem{prop}[thm]{Proposition}
\numberwithin{equation}{section} \theoremstyle{definition}
\newtheorem{defn}[thm]{Definition}

\newtheorem{eg}[thm]{Example}

\def\textmatrix#1&#2\\#3&#4\\{\bigl({#1 \atop #3}\ {#2 \atop #4}\bigr)}
\def\dispmatrix#1&#2\\#3&#4\\{\left({#1 \atop #3}\ {#2 \atop #4}\right)}

\begin{document}
	
	\title[Function theory of the hexablock]{Function theory of the hexablock and applications to the tetrablock and Euclidean biball}
	\author{SOURAV PAL AND NITIN TOMAR}
	
	\address[Sourav Pal]{Mathematics Department, Indian Institute of Technology Bombay,
		Powai, Mumbai - 400076, India.} \email{sourav@math.iitb.ac.in}
	
		\address[Nitin Tomar]{Statistics and Mathematics unit, Indian Statistical Institute Bangalore, Karnataka - 560059, India.} \email{tomarnitin414@gmail.com, nitin\_pd@isibang.ac.in}

	

	\begin{abstract}
		The realization, interpolation, extension and Toeplitz corona problems are amongst the central themes in function theory of a domain in $\mathbb{C}^d$. The present work addresses these four problems for the hexablock $\mathbb{H}$, a domain in $\mathbb{C}^4$ arising in connection with a special case of $\mu$-synthesis in $H^{\infty}$ control theory. We determine Schur-Agler class $SA(\mathbb{H})$ for $\mathbb{H}$ and find a realization formula for $\mathbb H$. With the help of this realization formula we state and prove interpolation, extension and Toeplitz corona theorems for the hexablock. As an application, we obtain analogous theorems for the Euclidean unit ball in $\mathbb{C}^2$. Moreover, we recover the existing same results for the tetrablock $\E$, another domain associated with the $\mu$-synthesis, as consequences of the hexablock theory and in terms of the Schur-Agler class $SA(\mathbb{H})$ and admissible kernels on $\mathbb{H}$.
	\end{abstract}	
	
	\maketitle

	\section{Introduction}
	
	\noindent Throughout the article, all operators are bounded linear maps on complex separable Hilbert spaces. Let $\mathcal{B}(\HS)$ be the algebra of operators on a Hilbert space $\HS$. We write $\D$ and $\T$ for the unit disc and unit circle in the complex plane $\C$, respectively, with center at the origin. Let $\sigma_T(\underline{T})$ denote the Taylor joint spectrum of a commuting tuple $\underline{T} = (T_1, \dotsc, T_d)$ of operators. Let $\mathrm{Hol}(\Omega)$ be the space of all holomorphic functions on a domain $\Omega$ in $\C^d$. We denote the Schur class of $\Omega$ by $S(\Omega) = \{ f \in \mathrm{Hol}(\Omega) : \sup_{z \in \Omega} |f(z)| \leq 1 \}$, and $H^\infty(\Omega)$ denotes the space of all bounded holomorphic functions on $\Omega$. The hexablock is a domain in $\C^4$ given by
	\[
	\Hex= \left\{(z^{(1)}, z^{(2)}, z^{(3)}, z^{(4)}) \in \C \times \E :
	\sup_{\alpha_1, \alpha_2 \in \D}
	\left|\frac{z^{(1)}\sqrt{(1-|\alpha_1|^2)(1-|\alpha_2|^2)}}{1-z^{(2)}\alpha_1-z^{(3)}\alpha_2+z^{(4)}\alpha_1\alpha_2}\right|<1 \right\},
	\]
	where $\E=\{(z^{(2)}, z^{(3)}, z^{(4)}) \in \C^3 : \ 1-z^{(2)}\alpha_1-z^{(3)}\alpha_2+z^{(4)}\alpha_1\alpha_2 \ne 0 \ \text{for all} \ \alpha_1, \alpha_2 \in \DC\}$. The domain $\E$ is referred to as the tetrablock. These two domains arise naturally in connection with the $\mu$-synthesis problem in control engineering. To explain this connection, we recall the notion of the structured singular value from \cite{Doyle, Francis}. Let $E$ be a linear subspace of $M_d(\C)$, the space of $d \times d$ complex matrices. The structured singular value of $A \in M_d(\C)$ is defined by
	\[
	\mu_E(A) = 1\slash \inf \{ \|X\| : X \in E,\ \det(I - AX) = 0 \},
	\]
	and $\mu_E(A) = 0$ if there does not exist $X \in E$ with $\det(I - AX) =0$ for all $X \in E$. The notion of $\mu_E$ comes up naturally in the study of systems with structured uncertainty and is important in the $\mu$-synthesis problem. Given distinct points $\alpha_1 , \dots , \alpha_n \in \D$ and matrices $B_1, \dots , B_n \in M_d(\C)$, the $\mu$-synthesis problem asks for an analytic function $F : \D \to M_d(\C)$ that interpolates the given data, i.e., $F(\alpha_i)=B_i$ for $1\leq i \leq n$ and satisfies $\mu_E(F(\lambda)) \le 1$ for all $\lambda \in \D$. Several well-known interpolation problems arise as special cases of the $\mu$-synthesis problem. For example, if $E = M_d(\C)$, then $\mu_E$ is the operator norm and the $\mu$-synthesis problem becomes the matricial Nevanlinna-Pick interpolation problem \cite{Foias_Frazho}. Also, if $E = \{\alpha I : \alpha \in \C\}$, then $\mu_E$ becomes the spectral radius and the corresponding $\mu$-synthesis problem reduces to the spectral Nevanlinna-Pick interpolation problem \cite{AglerI, Costara2005}. Interestingly, exploring different cases of the $\mu$-synthesis problem leads to different domains in $\C^d$. For instance, when $E$ is the linear subspace of all $2 \times 2$ scalar matrices, the associated $\mu$-synthesis problem gives rise to the symmetrized bidisc given by
	\[
	\G_2=\{(\lambda_1 + \lambda_2, \lambda_1 \lambda_2) \in \C^2 : |\lambda_1|, |\lambda_2| < 1\}.
	\]
	If $E$ is the space of all diagonal matrices in $M_2(\C)$, then one obtains the tetrablock $\E$. Also, when $E$ is taken to be the linear subspace of $2 \times 2$ upper triangular matrices with equal diagonal entries, the resulting domain is the pentablock given by
	\[
	\mathbb{P} = \{(a_{21}, \mathrm{tr}(A), \det(A)) \in \C^3 : A = [a_{ij}] \in M_2(\C), \|A\| < 1\}.
	\]
	The domains $\G_2, \E$ and $\Pe$ were introduced in \cite{AglerYoung}, \cite{Abouhajar} and \cite{AglerIV}, respectively. For more details on the connections of these three domains with $\mu$-synthesis, we also refer the readers to Chapter 1 of \cite{Biswas}. Besides from complex-geometric and function-theoretic point of view, the domains $\G_2$, $\E$ and $\Pe$ have been extensively studied from operator theoretic perspective in \cite{AglerYoung2000, AglerYoung2003, Tirtha_Pal_Roy}, \cite{Tirtha} and \cite{Pal_penta}, respectively. Motivated by the origins of these three domains, the authors of \cite{Biswas} considered the linear subspace $E$ consisting of all upper triangular matrices in $M_2(\C)$, which gives rise to the domain $\Hex$. To see this, let us denote by $\mu_{\text{hexa}}$ the corresponding function $\mu_E$ and define
	\[
	\Hex_\mu=\{(a_{21}, a_{11}, a_{22}, \det(A)) : A=[a_{ij}] \in M_2(\C),  \mu_{\text{hexa}}(A)<1\}.
	\]
	The set $\Hex_\mu$ is called $\mu$-hexablock by the authors of \cite{Biswas}. Unlike the $\mu$-hexablock $\Hex_\mu$, the hexablock $\Hex$ is a domain in $\C^4$ that satisfies $\Hex=\text{int}(\overline{\Hex}_\mu)$. The set $\Hex$ is referred to as the hexablock, motivated by the geometry of $\Hex \cap \mathbb{R}^4$, which is a bounded set with boundary consisting of four extreme sets and two hypersurfaces. 
	
	\smallskip 	In this article, we establish realization, interpolation, extension and Toeplitz corona theorems for the domain $\Hex$. As applications of our main results, we obtain analogous theorems for the biball $\B_2$, and we also recover the corresponding results for the tetrablock $\E$ established in \cite{Jain}. However, our characterizations of the realization, interpolation, extension and Toeplitz corona theorems for $\E$ are formulated in terms of admissible kernels on $\Hex$, thereby providing a unified framework for these results on both $\B_2$ and $\E$. These four problems arise naturally in the function theory of domains in $\C^d$ and have been studied extensively over the past several decades. We begin with a brief overview of these topics and recall some of the domains in $\C^d$ for which the corresponding results are known. The study of interpolation on domains in $\C^d$ typically begins with a characterization of the class of functions to which the interpolating functions belong.	Such a characterization is commonly referred to as a realization theorem. For the unit disc $\D$, the realization theorem (see \cite{Agler_McCarthy}) states that $f \in S(\D)$ if and only if there exist a Hilbert space $\HS$ and a unitary operator
	\[
	V=\begin{bmatrix} A & B \\ C & D \end{bmatrix} : \C \oplus \HS \to \C \oplus \HS
	\quad 
	\text{such that} 
	\quad 
	f(z)=A+zB(I_\HS-zD)^{-1}C.
	\]
	Realization formulas of this type have been established for several domains, including the annulus \cite{Drit2007_I} and the bidisc \cite{Agler1990, Agler_McCarthy}. In higher dimensions, such as the polydisc $\D^d$ for $d \geq 3$, not every Schur class function admits a realization. However, an important subclass, known as the Schur-Agler class, admits such a realization theorem \cite{Agler_McCarthy}. This framework has been further generalized to abstract settings where the domain is replaced by a set equipped with a family of test functions \cite{Drit2007_I, Drit2007_II, Ball_Guerra}. Realization theorems have also been obtained for domains arising from $\mu$-synthesis, such as the symmetrized bidisc \cite{AglerYoung2017, Tirtha_Sau}, the tetrablock \cite{Jain} and the pentablock \cite{Pal2026}. Once a realization theorem is established, one can address an interpolation problem on a domain in $\C^d$. The classical Nevanlinna-Pick interpolation theorem \cite{Nevanlinna, Pick} states that given distinct points $z_1, \dotsc, z_n$ in $\D$ and points $\lambda_1, \dotsc, \lambda_n$ in $\overline{\D}$, there exists a holomorphic function $f : \D \to \overline{\D}$ satisfying $f(z_i)=\lambda_i$ for $1 \leq i \leq n$ if and only if the Pick matrix
	\[
	\left[\frac{1-\lambda_i\overline{\lambda}_j}{1-z_i\overline{z}_j}\right]_{i,j=1}^n
	\]
	is positive semi-definite. The above positivity condition of the Pick matrix is same as the positive semi-definiteness of $\left[(1-\lm_i\overline{\lm}_j)k_S(z_i, z_j)\right]_{i, j=1}^n$, where $k_S(z, w)=(1-z\overline{w})^{-1}$ is the Szeg\H{o} kernel on $\D$. For more general domains, interpolation problems often require families of admissible kernels, as seen in the work of Abrahamse \cite{AbrahamseI} for multiply connected domains and Agler-McCarthy for the bidisc \cite{Agler_McCarthyI, Agler_McCarthy}. In this direction, one considers the following interpolation problem: Let $\Omega \subset \C^d$ be a domain. Given distinct points $z_1, \dotsc, z_n$ in $\Omega$ and points $\lm_1, \dotsc, \lm_n$ in $\overline{\D}$, when does there exist a holomorphic function $f : \Omega \to \C$ with $f(z_i)=\lm_i$ for $1 \leq i \leq n$? One usually imposes additional norm conditions on $f$. Such problems have been studied for several domains in the literature. An interested reader is referred to \cite{Deepak_Jaydeb, Kosinski_2015} for the polydisc $\D^d$ and \cite{Bhatt_Deepak_Jaydeb, Deepak_Jaydeb_II, Esch_Put, Kosinski_II} for the Euclidean unit ball $\mathbb{B}_d$. In addition, various interpolation problems in the class of Schur multipliers on $\B_d$ have been extensively studied in \cite{Ball_2003, Bolotnikov, Popescu}. The interpolation theorems for $\G_2$, $\E$ and $\Pe$ were presented in \cite{AglerYoung2017, Tirtha_Sau}, \cite{Jain} and \cite{Pal2026}, respectively. Finally, interpolation results often lead naturally to extension theorems. Let $V$ be a subset of a domain $\Omega \subset \C^d$. A function $f : V \to \C$ is said to be holomorphic if it extends to a holomorphic function in a neighbourhood of $V$. The norm-preserving extension problem asks for necessary and sufficient conditions under which every bounded holomorphic function on $V$ admits an extension to a function in $H^\infty(\Omega)$ without increasing its norm. Such extension results are established for several domains including the bidisc \cite{Agler_McCarthy_2003}, the symmetrized bidisc \cite{Tirtha_Sau}, the tetrablock \cite{Jain} and the pentablock \cite{Pal2026}. 
	
	\smallskip 
	
	Another important application of realization theorems for domains in $\C^d$ is the study of Toeplitz corona theorems associated with them. The classical corona problem for the Banach algebra $H^\infty(\D)$ was solved by Carleson \cite{Carleson}, who proved that if $\varphi_1,\dotsc,\varphi_d \in H^\infty(\D)$ satisfy
	\begin{equation}\label{eqn_corona}
		 \sup_{z \in\D}\left[|\varphi_1(z)|^2+\dotsc+|\varphi_n(z)|^2\right] \geq \epsilon^2
	\end{equation}
	for some $\epsilon>0$, then there exist $f_1,\dotsc,f_d \in H^\infty(\D)$ such that
	$
	\varphi_1f_1+\dotsc+\varphi_df_d=1.
	$
	Arveson \cite{Arveson_TC} established an operator-theoretic analogue, now known as the Toeplitz corona theorem, by replacing  \eqref{eqn_corona} with the operator inequality
	$
	T_{\varphi_1}T_{\varphi_1}^*+\dotsc+T_{\varphi_d}T_{\varphi_d}^*\geq \epsilon^2I.
	$
	Here,  $T_{\varphi_i}$ is the Toeplitz operator with symbol $\varphi_i$ for $1 \leq i \leq d$. Toeplitz corona theorems have been established for several domains including the bidisc \cite{Agler_McCarthyI}, the Euclidean unit ball and the polydisc \cite{Amar}, the symmetrized bidisc \cite{Tirtha_Sau_II}, the tetrablock \cite{Jain} and the pentablock \cite{Pal_penta_TC}.
		
	\smallskip 
	
	The principal domains of interest in this article are the hexablock $\Hex$, the tetrablock $\E$ and the biball $\B_2$. We present realization, interpolation, extension and Toeplitz corona theorems for $\Hex$ in Section \ref{sec_hexa}.  To begin with, we introduce the class $\mathfrak{M}_\Hex$ of commuting operator quadruples $\underline{T}=(T_1, T_2, T_3, T_4)$ satisfying $\|T_3\|<1, \|(\al T_4-T_2)(\al T_3-I)^{-1}\|<1$ and 
	\[
	\|\sqrt{(1-|\al_1|^2)(1-|\al_2|^2)}T_1(I-\al_1T_2-\al_2T_3+\al_1\al_2T_4)^{-1}\| <1
	\]
	for all $\al \in \DC^2$ and $(\al_1, \al_2) \in \DC^2$. The motivation behind considering $\mathfrak{M}_\Hex$ comes from the definition of $\Hex$, which is explained in Section \ref{sec_hexa}. Also, any operator quadruple in $\mathfrak{M}_\Hex$ has its joint spectrum contained in $\Hex$ and so, one can define the functional calculus $f(\underline{T})$ for every $\underline{T} \in \mathfrak{M}_\Hex$ and $f \in \text{Hol}(\Hex)$. In light of this observation, we define the Schur-Agler class for $\Hex$ as the class given by
	\[
	SA(\Hex)=\left\{f \in \text{Hol}(\Hex) : \ \text{$\|f(\underline{T})\| \leq 1$ for all $\underline{T} \in \mathfrak{M}_\Hex$}\right\}.
	\]
	Our realization theorem for $\Hex$ is precisely the characterization of functions in $SA(\Hex)$ stated as in Theorem \ref{thm_realization_H}. As mentioned earlier, realization theorem plays a pivotal role to obtain a corresponding Nevanlinna-Pick type interpolation theorem. In the same spirit, we capitalize the realization theorem for Schur-Agler class $SA(\Hex)$ and obtain in Theorem \ref{thm_interpolation_H} necessary and sufficient conditions for an interpolation problem on $\Hex$, where the interpolating function belongs to $SA(\Hex)$. Next, we turn our attention to the extension problem on $\Hex$. To do so, we introduce in Section \ref{sec_hexa} the class $Q\Hex$ consisting of all commuting quadruples $\underline{T}=(T_1, T_2, T_3, T_4)$ of Hilbert space operators with $\sigma_T(\underline{T}) \subseteq \Hex$ such that $\|T_3\| \leq 1, \|(\al T_4-T_2)(\al T_3-I)^{-1}\| \leq 1$ and 
	\[
	\|\sqrt{(1-|\al_1|^2)(1-|\al_2|^2)}T_1(I-\al_1T_2-\al_2T_3+\al_1\al_2T_4)^{-1}\| \leq 1
	\]
	for all $\al \in \DC^2$ and $(\al_1, \al_2) \in \DC^2$. Following the terminologies in  \cite{Mittal}, we refer to $Q\Hex$ as the quantum hexablock. By quantization, we mean that scalar variables are replaced by commuting operators satisfying analogous norm inequalities. Finally, we define the notion of subordination to a subset $W$ of $\Hex$ and consider the operator quadruples in $Q\Hex$ that are subordinate to $W$. With these notions in place, we present our extension theorem as Theorem \ref{thm_ext_H}. In particular, it is proved that if $W \subseteq \Hex$ and $f$ is a bounded function on $W$ that is holomorphic on $W$, then there exists $g \in H^\infty(\Hex)$ with $g|_W = f$, $\|g\|_{\infty,\Hex} = \|f\|_{\infty, W}$ and $g/\|f\|_{\infty,W} \in SA(\Hex)$ if and only if  
	$
	\|f(\underline{T})\| \le \|f\|_{\infty,W}
	$ 
	for every $\underline{T} \in Q\Hex$ subordinate to $W$. We also establish Toeplitz corona theorem for $\Hex$ as Theorem \ref{thm_TC_H_II}.

	\smallskip 
	
	Next, we present in Section \ref{sec_biball} the realization, interpolation, extension and Toeplitz corona type theorems on the biball $\B_2=\{(z^{(1)}, z^{(2)}) \in \C^2 : |z^{(1)}|^2+ |z^{(2)}|^2<1 \}$ as applications of the corresponding theorems associated with $\Hex$. The starting point in this direction is the fact that $(z^{(1)}, z^{(2)}) \in \mathbb{B}_2$ if and only if $(z^{(1)}, z^{(2)}, 0, 0) \in \Hex$. Moreover, if $(z^{(1)}, z^{(2)}, z^{(3)}, z^{(4)}) \in \Hex$, then $(z^{(1)}, z^{(2)}) \in \mathbb{B}_2$. Consequently, the domain $\B_2$ can be embedded holomorphically in $\Hex$. We refer to Chapter 10 in \cite{Biswas} for further details on the connection of $\B_2$ with $\Hex$. Motivated by this, we consider the class 
	\begin{align*}
		\mathfrak{M}_{\B_2}&=\{(T_1, T_2): (T_1, T_2, 0, 0) \in \mathfrak{M}_\Hex\}.
	\end{align*}
	For an operator pair $(T_1, T_2) \in \mathfrak{M}_{\B_2}$, it follows that $\sigma_T(T_1, T_2) \subseteq \B_2$ and so, one can define $g(T_1, T_2)$ for all $g \in \text{Hol}(\B_2)$. We then consider a subclass of the Schur class $S(\B_2)$ give by 
	\[
	\mathfrak{C}(\B_2)=\{g \in \text{Hol}(\B_2): \|g(T_1, T_2)\|\leq 1 \ \text{for all} \ (T_1, T_2) \in \mathfrak{M}_{\B_2} \}.
	\]
	It is proved in Theorem \ref{thm_realization_H_B2} that $g \in \mathfrak{C}(\B_2)$ if and only if $g\circ \theta_{\Hex \to \B_2} \in SA(\Hex)$, where $\theta_{\Hex \to \B_2}$ is the map give by $\theta_{\Hex \to \B_2}(z^{(1)}, z^{(2)}, z^{(3)}, z^{(4)})=(z^{(1)}, z^{(2)})$. Consequently, the realization theorem for $SA(\Hex)$ yields a realization theorem for functions in $\mathfrak{C}(\B_2)$. In an analogous manner, we obtain interpolation extension and Toeplitz corona type results on $\B_2$ in Theorems \ref{thm_interpolation_H_B2}, \ref{thm_ext_H_B2} and \ref{thm_TC_B2}, respectively, from the corresponding theorems on $\Hex$. In Section \ref{sec_tetra}, we recover the realization, interpolation, extension and Toeplitz corona theorems on the tetrablock $\E$, where the realization theorems for functions in the Schur-Agler class $SA(\E)$ is now presented in terms of functions in $SA(\Hex)$. Similarly, the interpolation theorem is resolved in terms of functions and kernels associated with the hexablock framework. To begin with, note that if $(z^{(1)}, z^{(2)}, z^{(3)}, z^{(4)}) \in \Hex$, then $(z^{(2)}, z^{(3)}, z^{(4)}) \in \E$. Moreover, $(z^{(1)}, z^{(2)}, z^{(3)}) \in \E$ if and only if $(0, z^{(1)}, z^{(2)}, z^{(3)}) \in \Hex$. Therefore, the domain $\E$ can be holomorphically embedded inside $\Hex$. As an application of this, we prove in Theorem \ref{thm_realization_H_E} that $g \in SA(\E)$ if and only if $g\circ \theta_{\Hex \to \E} \in SA(\Hex)$, where $\theta_{\Hex \to \E}(z^{(1)}, z^{(2)}, z^{(3)}, z^{(4)})=(z^{(1)}, z^{(2)}, z^{(3)})$. This provides a realization theorem for functions in $SA(\E)$ through the realization theorem for functions in $SA(\Hex)$ presented in Section \ref{sec_hexa}. In a similar way, we present in Theorems \ref{thm_interpolation_H_E}, \ref{thm_ext_H_E} and \ref{thm_TC_E} the interpolation, extension and Toeplitz corona theorems on $\E$, respectively, as applications of the corresponding results on $\Hex$ proved in Section \ref{sec_hexa}.
	
	\section{Realization, interpolation, extension and Toeplitz corona theorems on $\Hex$}\label{sec_hexa}
	
	\noindent In this section, we introduce the Schur-Agler class for the hexablock $\Hex$ and establish a realization theorem for functions in this class. Also, we prove an interpolation theorem on $\Hex$ with interpolating functions in the Schur-Agler class for $\Hex$. Recall that the hexablock is given by
	\[
	\mathbb H= \left\{(z^{(1)}, z^{(2)}, z^{(3)}, z^{(4)}) \in \mathbb{C} \times \mathbb{E} :
	\sup_{\alpha_1, \alpha_2 \in \mathbb{D}}
	\left|\frac{z^{(1)}\sqrt{(1-|\al_1|^2)(1-|\al_2|^2)}}{1-z^{(2)}\al_1-z^{(3)}\al_2+z^{(4)}\al_1\al_2}\right|<1 \right\},
	\]
	where $\E$ itself is a domain in $\C^3$ defined as
	\[
	\E=\left\{(z^{(2)}, z^{(3)}, z^{(4)}) \in \C^3 : 1-z^{(2)}\al_1-z^{(3)}\al_2+z^{(4)}\al_1\al_2 \neq 0 \ \text{for all} \ \al_1, \al_2 \in \overline{\D}\right\}.
	\]
	The domain $\E$ is called the tetrablock. The authors of \cite{Abouhajar} gave the following characterization of $\E$:
	\begin{equation}\label{eqn_E}
		(z^{(2)}, z^{(3)}, z^{(4)}) \in \E \ \  \text{if and only if}	\ |z^{(3)}|<1  \ \text{and} \ |\Psi(\al, z^{(2)}, z^{(3)}, z^{(4)})|<1 \ \text{for all} \ \al \in \DC, 
	\end{equation}
	where $\Psi(\al, z^{(2)}, z^{(3)}, z^{(4)})=(\al z^{(4)}-z^{(2)})\slash (\al z^{(3)}-1)$. Capitalizing the above description of $\E$, the authors of \cite{Biswas} proved that $z=(z^{(1)}, z^{(2)}, z^{(3)}, z^{(4)}) \in \Hex$ if and only if 
	\begin{equation}\label{eqn_H}
		|z^{(3)}|<1, \quad |\Psi(\al, z^{(2)}, z^{(3)}, z^{(4)})|<1 \quad \text{and} \quad |\psi_{\al_1, \al_2}(z^{(1)}, z^{(2)}, z^{(3)}, z^{(4)})|<1
	\end{equation}
	for all $\al, \al_1, \al_2 \in \DC$, where 
	\[
	\psi_{\al_1, \al_2}(z^{(1)}, z^{(2)}, z^{(3)}, z^{(4)})=\frac{z^{(1)}\sqrt{(1-|\al_1|^2)(1-|\al_2|^2)}}{1-z^{(2)}\al_1-z^{(3)}\al_2+z^{(4)}\al_1\al_2}.
	\]
	An interesting fact about the domains $\E$ and $\Hex$ is that both are quasi-balanced domains, i.e., $(rw^{(2)}, rw^{(3)}, r^2w^{(4)}) \in \E$ and $(rz^{(1)}, rz^{(2)}, rz^{(3)}, r^2z^{(4)}) \in \Hex$ for $0 \leq r \leq 1, (w^{(2)}, w^{(3)}, w^{(4)}) \in \E$ and $(z^{(1)}, z^{(2)}, z^{(3)}, z^{(4)}) \in \Hex$. Motivated by the description of $\Hex$ as in \eqref{eqn_H}, we define for every point $z = (z^{(1)}, z^{(2)}, z^{(3)}, z^{(4)}) \in \Hex$ the following functions:
	\begin{align}\label{eqn_E(z)}
		&E(z): \DC^2 \to \C, \quad E(z)(\al_1, \al_2)=\psi_{\al_1, \al_2}(z^{(1)}, z^{(2)}, z^{(3)}, z^{(4)})=\frac{z^{(1)}\sqrt{(1-|\al_1|^2)(1-|\al_2|^2)}}{1-z^{(2)}\al_1-z^{(3)}\al_2+z^{(4)}\al_1\al_2}; \notag \\
		& e(z): \DC \to \C,  \qquad e(z)(\al) =\Psi(\al, z^{(2)}, z^{(3)}, z^{(4)})=\frac{\al z^{(4)}-z^{(2)}}{\al z^{(3)}-1}.
	\end{align}
	It is evident that $E(z)$ and $e(z)$ are continuous functions on $\DC^2$ and $\DC$, respectively, for every $z \in \Hex$. Furthermore, the supremum norms $\|E(z)\|_{\infty, \DC^2}, \|e(z)\|_{\infty, \DC}<1$ for every $z \in \Hex$. We now introduce an operator-theoretic analog of the inequalities provided in \eqref{eqn_H}. Let $(T_1, T_2, T_3, T_4)$ be a commuting quadruple of operators acting on a Hilbert space $\HS$, and let $\|T_3\|<1$. Evidently, the operator $(\alpha T_3-I)$ is invertible for all $\al \in \DC$. Additionally, assume that
	\[
	\|\Psi(\al, T_2, T_3, T_4)\|=\|(\al T_4-T_2)(\al T_3-I)^{-1}\|<1 
	\]
	for all $\al \in \DC$. The class of commuting triples of operators given by 
	\[
	\mathfrak{M}_\E=\left\{(T_2, T_3, T_4): \|T_3\|<1, \|\Psi(\al, T_2, T_3, T_4)\|<1 \ \text{for all} \ \al \in \DC\right\}
	\]
	was introduced by the authors of \cite{Jain}. The inequalities in $\mathfrak{M}_\E$ may be viewed as an operator theoretic analog of the inequalities in \eqref{eqn_E}, where the scalars in $\E$ are replaced in \eqref{eqn_E} by commuting operator triples satisfying similar norm bounds. The class $\mathfrak{M}_\E$ contains $\E$ in the sense that $(z^{(2)}I, z^{(3)}I, z^{(4)}I) \in \mathfrak{M}_\E$ for all $(z^{(2)}, z^{(3)}, z^{(4)}) \in \E$, which follows from \eqref{eqn_E}. 
	
	\smallskip 
	
	We now show that the joint spectrum of an operator triple in $\mathfrak{M}_\E$ is contained in $\E$. This inclusion has been established in \cite{Jain}. We provide an alternative proof here. Let $(T_2, T_3, T_4) \in \mathfrak{M}_\E$ and let $(w^{(2)}, w^{(3)}, w^{(4)}) \in \sigma_T(T_2, T_3, T_4)$. By projection property of the joint spectrum, it follows that $w^{(3)} \in \sigma(T_3)$ and so, $|w^{(3)}|<1$ since $\|T_3\|<1$. For $\al \in \DC$, we have that $\|\Psi(\al, T_2, T_3, T_4)\|<1$. It now follows from spectral mapping principle that
	\[
	\Psi(\al, w^{(2)}, w^{(3)}, w^{(4)}) \in \Psi(\alpha, \sigma_T(T_2, T_3, T_4))=\sigma(\Psi(\al, T_2, T_3, T_4))\subseteq \D.
	\]
	By \eqref{eqn_E}, $(w^{(2)}, w^{(3)}, w^{(4)}) \in \E$ and $\sigma_T(T_2, T_3, T_4) \subseteq \E$. So, one can define $f(T_2, T_3, T_4)$ for every $f \in \text{Hol}(\E)$ and $(T_2, T_3, T_4) \in \mathfrak{M}_\E$ via the holomorphic functional calculus. The Schur-Agler class for $\E$ is defined as 
	\[
	SA(\E)=\{f \in \text{Hol$(\E)$} : \|f(T_2, T_3, T_4)\| <1 \ \text{for all} \ (T_2, T_3, T_4) \in \mathfrak{M}_\E \}.
	\]
	The reader is referred to \cite{Jain} for further details on $\mathfrak{M}_\E$ and $SA(\E)$, where these two classes were introduced. Next, for $(\al_1, \al_2) \in \DC^2$, consider the polynomial 
	\[
	p(z^{(2)}, z^{(3)}, z^{(4)})=1-z^{(2)}\al_1-z^{(3)}\al_2+z^{(4)}\al_1\al_2.
	\]
	By the definition of $\E$, the function $p$ does not vanish on $\E$ and consequently, it is non-vanishing on $\sigma_T(T_2, T_3, T_4)$ for every $(T_2, T_3, T_4) \in \mathfrak{M}_\E$. Again by spectral mapping principle, the operator 
	$
	p(T_2, T_3, T_4)=I-\al_1T_2-\al_2T_3+\al_1\al_2T_4
	$
	is invertible for every $(\al_1, \al_2) \in \DC^2$. Consequently, one can define the operator 
	\[
	\psi_{\al_1, \al_2}(T_1, T_2, T_3, T_4)=\sqrt{(1-|\al_1|^2)(1-|\al_2|^2)}T_1(I-\al_1T_2-\al_2T_3+\al_1\al_2T_4)^{-1}
	\]
	for all $(\al_1, \al_2) \in \DC^2$ and commuting quadruples $(T_1, T_2, T_3, T_4)$ of Hilbert space operators such that $(T_2, T_3, T_4) \in \mathfrak{M}_\E$. In view of this, we introduce the class $\mathfrak{M}_\Hex$ consisting of all commuting quadruples $\underline{T}=(T_1, T_2, T_3, T_4)$ of Hilbert space operators such that 
	\[
	\|T_3\| <1, \quad \|\Psi(\al, T_2, T_3, T_4)\|<1 \quad \text{and} \quad \|\psi_{\al_1, \al_2}(T_1, T_2, T_3, T_4)\|<1
	\]
	for all $\al, \al_1, \al_2 \in \DC$. We note the following basic properties of $\mathfrak{M}_\Hex$, which will be used subsequently in this section.
	\begin{enumerate}[leftmargin=*]
		
		\item $\mathfrak{M}_\Hex$ contains $\Hex$ in the sense that $(z^{(1)}I, z^{(2)}I, z^{(3)}I, z^{(4)}I) \in \mathfrak{M}_\Hex$ for all $(z^{(1)}, z^{(2)}, z^{(3)}, z^{(4)}) \in \Hex$.
		
		\item It is evident to see that $(T_1, T_2, T_3, T_4) \in \mathfrak{M}_\Hex$ if and only if $(T_1^*, T_2^*, T_3^*, T_4^*) \in \mathfrak{M}_\Hex$.
		
		\item Since $\Hex$ is a $(1, 1, 1, 2)$-quasi-balanced domain, it follows that 
		$(rT_1, rT_2, rT_3, r^2T_4) \in \mathfrak{M}_\Hex$ for $0 \leq r \leq 1$ and $(T_1, T_2, T_3, T_4) \in \mathfrak{M}_\Hex$.
		
		\item For a commuting operator quadruple $(T_1, T_2, T_3, T_4) \in \mathfrak{M}_\Hex$, it is evident that $(T_2, T_3, T_4) \in \mathfrak{M}_\E$.
		
		\item It is easy to verify that $\sigma_T(T_1, T_2, T_3, T_4) \subseteq \Hex$ for all $(T_1, T_2, T_3, T_4) \in \mathfrak{M}_\Hex$. To see this, let $\underline{T}=(T_1, T_2, T_3, T_4) \in \mathfrak{M}_\Hex$ and let $(w^{(1)}, w^{(2)}, w^{(3)}, w^{(4)}) \in \sigma_T(\underline{T})$. By projection property of the joint spectrum, $(w^{(2)}, w^{(3)}, w^{(4)}) \in \sigma_T(T_2, T_3, T_4)$. Since $(T_2, T_3, T_4) \in \mathfrak{M}_\E$, it follows that $(w^{(2)}, w^{(3)}, w^{(4)}) \in \E$. For $\alpha_1, \alpha_2 \in \DC$, we have that $\|\psi_{\alpha_1, \alpha_2}(T_1, T_2, T_3, T_4)\|<1$. By spectral mapping mapping principle, we have
		\[
		\psi_{\alpha_1, \alpha_2}(w^{(1)}, w^{(2)}, w^{(3)}, w^{(4)}) \in \psi_{\alpha_1, \alpha_2}(\sigma_T(T_1, T_2, T_3, T_4))=\sigma(\psi_{\alpha_1, \alpha_2}(T_1, T_2, T_3, T_4)) \subseteq \D 
		\]
		and by \eqref{eqn_H}, $(w^{(1)}, w^{(2)}, w^{(3)}, w^{(4)}) \in \Hex$. This ensures that the functional calculus $f(\underline{T})$ is well-defined for any $f \in \text{Hol}(\Hex)$ and $\underline{T} \in \mathfrak{M}_\Hex$.
	\end{enumerate}
	Accordingly, we now introduce the Schur-Agler class for $\Hex$, a subclass of $S(\Hex)$, defined by
	\[
	SA(\Hex)=\left\{f \in \text{Hol}(\Hex) : \|f(\underline{T})\| <1 \ \text{for all} \ \underline{T} \in \mathfrak{M}_\Hex\right\}.
	\]
	
	\begin{defn} For a non-empty set $Y$, a map $f: Y \times Y \to \C$ is said to be positive semi-definite if
		$
		\sum_{i,j=1}^n \overline{c}_ic_j f(y_i, y_j) \geq 0$
		for every $\{y_1,\dotsc, y_n\} \subseteq Y, \{c_1,\ldots,c_n\} \subset \C$ and $n \in \N$.  A positive semi-definite function $k: \Hex \times \Hex \to \C$ is referred to as a weak kernel on $\Hex$. In addition, if $k(z, z) \ne 0$ for all $z \in \Hex$, we say that $k$ is a kernel on $\Hex$. Let $F$ be a subset of $\Hex$ and $K \subseteq \C^m$ be compact. A function $\xi: F \times F \to C(K)^*$ is said to be a positive kernel with values in the dual $C(K)^*$ of the space $C(K)$ of complex-valued continuous functions  on $K$ if the following holds:
		\[
		\overset{n}{\underset{i, j=1}{\sum}}\overline{c}_ic_j\xi(z_i, z_j)(f_i \overline{f}_j) \geq 0
		\]  
		for all $n \in \N, \{z_1, \dotsc, z_n\} \subseteq F, \{c_1, \dotsc, c_n\} \subset \C$ and $\{f_1, \dotsc, f_n\} \subseteq C(K)$. A kernel (weak kernel) $k: \Hex \times \Hex \to \C$ is said to be \textit{admissible} (\textit{weakly admissible}) if the following hold:
		\begin{enumerate}[leftmargin=*]
			\item $\left(1-z^{(3)}\overline{w}^{(3)}\right)k(z, w) \succcurlyeq 0$; \smallskip 
			\item $\left(1-\Psi(\al, z^{(2)}, z^{(3)}, z^{(4)})\overline{\Psi(\al, w^{(2)}, w^{(3)}, w^{(4)})}\right)k(z, w) \succcurlyeq 0$ for all $\al \in \DC$; \smallskip 
			\item $\left(1-\psi_{\al_1, \al_2}(z^{(1)}, z^{(2)}, z^{(3)}, z^{(4)})\overline{\psi_{\al_1, \al_2}(w^{(1)}, w^{(2)}, w^{(3)}, w^{(4)}})\right)k(z, w) \succcurlyeq 0$ for all $(\al_1, \al_2) \in \DC^2$.
		\end{enumerate}
		
	\end{defn} 
	
	We write $f \succcurlyeq 0$ to denote that $f$ is a positive semi-definite function. We denote by $C(K)_F^+$ the set of all positive kernels with values in $C(K)^*$, and the set of all positive semi-definite functions $\delta: F \times F \to \C$ is denoted by $\C_F^+$. The class of admissible kernels on $\Hex$ is denoted by $AK(\Hex)$. With these definitions in place, we now present the following result. 
	
	\begin{prop}\label{prop_prelim_I_H}
		Let $F$ be a subset of $\Hex$ and let $K$ be a compact subset of $\C^n$. If $\xi \in C(K)_F^+$, then there exist a Hilbert space $\HS$, a function $L: F \to \mathcal{B}(C(K),\HS)$ and a unital $*$-representation $\rho: C(K) \to \mathcal{B}(\HS)$ such that
		\[
		\xi(z, w)(f\overline{g})=\langle L(z)f, L(w)g \rangle \quad \text{and} \quad L(z)(fg)=\rho(f)L(z)(g)
		\]	
		for all $f, g \in C(K)$ and $z, w  \in F$. 
	\end{prop}
	
	\begin{proof}
		Define the map $\xi': (F \times C(K)) \times (F \times C(K)) \to \C$ as $\xi'((z, h_1), (w, h_2))=\xi(z, w)(h_1\overline{h}_2)$. It is not difficult to see that $\xi'$ is a positive definite function on $F \times C(K)$. It follows from Theorem 2.53 in \cite{Agler_McCarthy} that there is a Hilbert space $\mathcal{H}$ and a vector-valued map $g: F \times C(K) \to \mathcal{H}$ such that 
		\[
		\langle g(z, h_1), g(w, h_2)\rangle_{\mathcal{H}}=\xi'((z, h_1), (w, h_2))=\xi(z, w)(h_1\overline{h}_2)
		\]
		for all $h_1, h_2 \in C(K)$ and  $z, w \in F$. Indeed, one can choose the Hilbert space $\HS=\overline{\text{span}}\{g(z, h): z \in F, h \in C(K)\}$. Let us define
		$L: F \to \mathcal{B}(C(K), \HS)$ as $L(z)(h)=g(z, h)$. Then 
		\[
		\xi(z, w)(f\overline{g})=\langle L(z)f, L(w)g \rangle \quad \text{and} \quad \|L(z)(h)\|^2=\|g(z, h)\|^2=\|\xi(z, z)(h\overline{h})\| \leq \|\xi(z, z)\| \|h\|^2_{\infty, K}
		\]
		for all $z \in F, h \in C(K)$. Consider the map $\rho: C(K) \to \mathcal{B}(\HS)$ given by $\rho(h_1)g(z, h_2)=g(z, h_1h_2)$. A simple calculation shows that $\rho$ is a unital $*$-representation such that $\rho(f)L(z)(g)=L(z)(fg)$.
	\end{proof}
	
	For a non-empty set $Y$, a map $f: Y \times Y \to \C$ is said to be \textit{self-adjoint} on $Y$ if $f(z, w)=\overline{f(w, z)}$ for all $z, w \in Y$. Our next result presents a description of self-adjoint functions on $\Hex$ such that their product with every admissible kernel is positive semi-definite. As before, we follow the methods developed in Section 3 of \cite{Drit2007_I} and Section 5 of \cite{Drit2007_II}. Let $F$ be a finite subset of $\Hex$ with cardinality $|F|$. Let $\mathfrak{C}_F$ be the collection of matrices of the form 
	\[
	\left[\xi(z, w)(1-E(z)\overline{E(w)})+\nabla(z, w)(1-e(z)\overline{e(w)})+(1-z^{(3)}\overline{w}^{(3)})\delta(z, w)\right]_{z, w \in F},
	\]
	where $\xi \in C(\overline{\D}^2)^+_F, \nabla \in C(\DC)^+_F$ and $\delta \in \C_F^+$. Following the same arguments as in Lemma 3.4 of \cite{Drit2007_I}, we observe that $\mathfrak{C}_F$ is a closed cone in $M_{|F|}(\C)$, the space of $|F| \times |F|$ matrices. Furthermore, any positive semi-definite matrix in $M_{|F|}(\C)$ belongs to $\mathfrak{C}_F^+$ and hence, has non-empty interior. To see this, let $P=[P(z, w)]_{z, w \in F}$ be a positive semi-definite matrix. Define $\delta(z, w)=P(z, w)\delta_0(z, w)$, where $\displaystyle \delta_0(z, w)=\frac{1}{1-z^{(3)}\overline{w}^{(3)}}$ for $z, w \in F$. A simple computation shows that $\delta \in \C_F^+$ and so, 
	\[
	[P(z, w)]_{z, w \in F}=[\delta(z, w)(1-z^{(3)}\overline{w}^{(3)})]_{z, w \in F} \in \mathfrak{C}_F.
	\]
	We refer the readers to \cite{Drit2007_I, Drit2007_II} for a more detailed discussion in a broader setting.
	
	\begin{thm}\label{thm_sa_H}
		Let $g: \Hex \times \Hex \to \C$ be a self-adjoint function. If $gk \succcurlyeq 0$ for all $k \in AK(\Hex)$, then there exist $\xi \in C(\overline{\D}^2)^+_\Hex, \nabla \in C(\DC)^+_\Hex$ and $\delta \in \C_\Hex^+$ such that 
		for all $z, w \in \Hex$,
		\[
		g(z, w)=\xi(z, w)(1-E(z)\overline{E(w)})+\nabla(z, w)(1-e(z)\overline{e(w)})+(1-z^{(3)}\overline{w}^{(3)})\delta(z, w).
		\] 
	\end{thm}
	
	\begin{proof}
		Let $gk \succcurlyeq 0$ for all $k \in AK(\Hex)$, and let $F=\{z_1, \dotsc, z_n\}$ be a subset of $\Hex$. We show that $G=[g(z_i, z_j)]_{i, j=1}^n \in \mathfrak{C}_F$. Let if possible, $G \notin \mathfrak{C}_F$. Since $\mathfrak{C}_F$ is a closed cone with non-empty interior, an application of Hahn-Banach theorem ensures the existence of a linear functional $L$ on $M_{n}(\C)$ such that $L(G)<0$ and $L(M) \geq 0$ for all $M \in \mathfrak{C}_F$. In fact, the functional $L$ can be chosen such that $L(A)=\text{tr}(AC)$ for all self-adjoint matrices $A$, where $C=[C_{ij}]_{i, j=1}^n$ is a fixed self-adjoint matrix. Given complex scalars $c_1, \dotsc, c_n$, we have that $N=[\overline{c}_ic_j] \in \mathfrak{C}_F$ and thus, it follows that
		\[
		L(N)=\overset{n}{\underset{i, j=1}{\sum}}\overline{c}_ic_jC_{ji} \geq 0.
		\]
		Consequently, the map $C_*: F \times F \to \C$ given by $C_*(z_i, z_j)=C_{ji}$ is positive semi-definite and so, $C_*$ is a weak kernel on $F$. Consider the matrices in $\mathfrak{C}_F$ given by
		\begin{align*}
			N_*&=\left[\overline{c}_ic_j(1-z_i^{(3)}\bar{z}_j^{(3)}\right]_{i, j=1}^n, \quad N_\alpha=\left[\overline{c}_ic_j(1-\Psi(\alpha, z_i^{(2)}, z_i^{(3)}, z_i^{(4)})\overline{\Psi(\alpha, z_j^{(2)}, z_j^{(3)}, z_j^{(4)})})\right]_{i, j=1}^n, \\
			N_{\al_1, \al_2}&=\left[\overline{c}_ic_j(1-\psi_{\al_1, \al_2}(z_i)\overline{\psi_{\al_1, \al_2}(z_j)})\right]_{i, j=1}^n
		\end{align*}
		for $\al, \al_1, \al_2 \in \DC$. An application of the positive semi-definiteness of $L(N_*), L(N_\al)$ and $L(N_{\al_1, \al_2})$ yields that $C_*$ is a weakly admissible kernel on $F$. One can define $C_*$ on $\Hex$ by putting $C_*(z, w)=0$ for all $(z, w) \notin F \times F$. For $\epsilon >0$, it is evident that $\epsilon k+C_* \succcurlyeq 0$ for all $k \in AK(\Hex)$. By hypothesis, $g(\epsilon k +C_*) \succcurlyeq 0$ on $\Hex$, and in particular, on $F$. Consequently, $gC_* \succcurlyeq 0$ on $F$. Then
		\[
		L(G)=\text{tr}(GC)=\overset{n}{\underset{i, j=1}{\sum}}g(z_i, z_j)C_*(z_i, z_j) \geq 0,
		\]
		which contradicts the fact that $L(G)<0$. Hence, $G \in \mathfrak{C}_F$ and thus, it follows that 
		there exist $\xi \in C(\overline{\D}^2)^+_F, \nabla \in C(\DC)^+_F$ and $\delta \in \C_F^+$ such that 
		\[
		g(z, w)=\xi(z, w)(1-E(z)\overline{E(w)})+\nabla(z, w)(1-e(z)\overline{e(w)})+(1-z^{(3)}\overline{w}^{(3)})\delta(z, w)
		\]
		for all $z, w \in F$. Thus, we obtain the desired representation of $g$ on every finite subset of $\Hex$. Since a representation on a larger set restricts naturally to any smaller subset, an application of the same reasoning as in Theorem 11.5 of \cite{Agler_McCarthy}, together with Kurosh’s theorem (see \cite{V_Arkh}, Page 7) gives the asserted conclusion.
	\end{proof}
	
	We now define a unitary colligation associated with $\Hex$.
	
	\begin{defn} A function $f: \Hex \to \C$ is said to be associated to a \textit{unitary colligation} if there exist  
		\begin{enumerate}
			\item Hilbert spaces $\HS_1, \HS_2$ and $\HS_3$,
			\item unital $*$-representations $\rho_1: C(\DC^2) \to \mathcal{B}(\HS_1), \rho_2: C(\DC) \to \mathcal{B}(\HS_2)$ and 
			\item a unitary $V=\begin{bmatrix} A & B \\ C & D \end{bmatrix}: \C \oplus \HS \to \C \oplus \HS$, where $\HS=\HS_1 \oplus \HS_2 \oplus \HS_3$
		\end{enumerate}
		such that $f(z)=A+BX(z)(I_\HS-DX(z))^{-1}C$, where
		\[
		X(z)=\begin{bmatrix}
			\rho_1(E(z)) & 0 & 0 \\
			0 & \rho_2(e(z)) & 0 \\
			0 & 0 & z^{(3)}I_{\HS_3}
		\end{bmatrix} \quad (z=(z^{(1)}, z^{(2)}, z^{(3)}, z^{(4)}) \in \Hex).
		\]
		The class of functions $f: \Hex \to \C$ associated to a unitary colligation is denoted by $UC(\Hex)$.
	\end{defn}
	It follows from the definition of $UC(\Hex)$ that $UC(\Hex) \subseteq S(\Hex)$. Furthermore, $\|X(z)\|<1$ since $\rho_1, \rho_2$ are unital $*$-representations and $\|E(z)\|_{\infty, \DC^2}, \|e(z)\|_{\infty, \DC}<1$ for all $z \in \Hex$.
	
	\begin{defn}\label{defn_simple}
		Given a compact set $K \subseteq \C^n$, a unital $*$-representation $\rho: C(K) \to \mathcal{B}(\HS)$ is said to be \textit{simple} if there exist $z_1, \dotsc, z_m \in K$ and orthogonal projections $P_1, \dotsc, P_m \in \mathcal{B}(\HS)$ with $P_1+\dotsc +P_m=I_\HS$ such that 
		$
		\rho(f)=f(z_1)P_1+\dotsc+f(z_m)P_m
		$
		for all $f \in C(K)$. 
	\end{defn}
	
	Whenever the unital $*$-representations in a unitary colligation of $f$ are simple, the following result shows that $f \in UC(\Hex)$ is in $SA(\Hex)$.
	
	\begin{lem}\label{lem_prelim_III_H}
		Let $f \in UC(\Hex)$, and let $\rho_1: C(\DC^2) \to \mathcal{B}(\HS_1), \rho_2: C(\DC) \to \mathcal{B}(\HS_2)$ be the associated unital $*$-representations. If $\rho_1$ and $\rho_2$ are simple representations, then $f \in SA(\Hex)$. 
	\end{lem}
	
	\begin{proof} Let $\underline{T}=(T_1, T_2, T_3, T_4) \in \mathfrak{M}_\Hex$ be defined on a Hilbert space $\mathcal{K}$. Assume that
		\[
		f(z)=A+BX(z)(I_\HS-DX(z))^{-1}C,
		\]
		where $\HS=\HS_1 \oplus \HS_2 \oplus \HS_3$, $V=\begin{bmatrix} A & B \\ C & D \end{bmatrix}: \C \oplus \HS \to \C \oplus \HS$ is a unitary and 
		\[
		X(z)=\begin{bmatrix}
			\rho_1(E(z)) & 0 & 0 \\
			0 & \rho_2(e(z)) & 0 \\
			0 & 0 & z^{(3)}I_{\HS_3}
		\end{bmatrix}
		\]
		for all $z=(z^{(1)}, z^{(2)}, z^{(3)}, z^{(4)}) \in \Hex$. Clearly, $f \in \text{Hol}(\Hex)$. Since $\rho_1$ is simple, there exist points $(\alpha_{1}^{(1)}, \alpha_{1}^{(2)}), \dotsc,  (\alpha_{m}^{(1)}, \alpha_{m}^{(2)}) \in \DC^2$ and orthogonal projections $P_1, \dotsc, P_m \in \mathcal{B}(\HS_1)$ summing to $I_{\HS_1}$. Since $\rho_2$ is also simple, one can find $\beta_1, \dotsc, \beta_\ell \in \DC$ and orthogonal projections $Q_1, \dotsc, Q_\ell \in \mathcal{B}(\HS_2)$ whose sum equals $I_{\HS_2}$. Furthermore, 
		\[
		\rho_1(\widetilde{f})=\overset{m}{\underset{j=1}{\sum}}\widetilde{f}(\alpha_{j}^{(1)}, \alpha_{j}^{(2)})P_j \quad \text{and} \quad \rho_2(g)=\overset{\ell}{\underset{j=1}{\sum}}g(\beta_j)Q_j
		\]
		for all $\widetilde{f} \in C(\DC^2)$ and $g \in C(\DC)$. Therefore, we have by \eqref{eqn_E(z)}that
		\[
		\rho_1(E(z))= \overset{m}{\underset{j=1}{\sum}}P_j \psi_{\al_j^{(1)}, \al_j^{(2)}}(z) \quad \text{and} \quad \rho_2(e(z))=\overset{\ell}{\underset{j=1}{\sum}}Q_j \Psi(\beta_j, z^{(2)}, z^{(3)}, z^{(4)}).
		\]
		It is not difficult to see that $\|X(\underline{T})\|\leq 1$. Using the representation $f(z)=A+BX(z)(I_\HS-DX(z))^{-1}C$ for all $z \in \Hex$, a simple computation shows that 
		\[
		f(\underline{T})=A\otimes I_{\mathcal{K}}+(B\otimes I_{\mathcal{K}})X(\underline{T})\left((I_\HS\otimes I_{\mathcal{K}})-(D\otimes I_{\mathcal{K}})X(\underline{T})\right)^{-1}(C\otimes I_{\mathcal{K}}).
		\]
		A laborious but routine calculation gives that $1-\overline{f(z)}f(z)=C^*(I-DX(z))^{-*}(I-X(z)^*X(z))(I-DX(z))^{-1}C$. For $R(\underline{T})=\left(I_{\HS \otimes \mathcal{K}}-(D\otimes I_{\mathcal{K}})X(\underline{T})\right)^{-1}(C\otimes I_{\mathcal{K}})$, we can write
		\[
		I_{\mathcal{K}}-f(\underline{T})^*f(\underline{T})=R(\underline{T})^*\left(I_{\HS\otimes \mathcal{K}}-X(\underline{T})^*X(\underline{T})\right)R(\underline{T}).
		\]
		Since $\|X(\underline{T})\| \leq 1$, we have $I_\mathcal{K}-f(\underline{T})^*f(\underline{T}) \geq 0$ and so, $\|f(\underline{T})\| \leq 1$. Hence, $f\in SA(\Hex)$. 
	\end{proof}
	
	We now present the realization theorem for functions belonging to the Schur-Agler class  $SA(\Hex)$.
	
	\begin{thm}\label{thm_realization_H}
		Let $f: \Hex \to \C$ be a function. Then the following are equivalent: 
		\begin{enumerate}[leftmargin=*]
			\item[$(1)$] $f \in SA(\Hex)$;
			\item[$(2)$] $(1-f(z)\overline{f(w)})k(z, w) \succcurlyeq 0$ for all $k \in AK(\Hex)$;
			\item[$(3)$] there exist $\xi \in C(\overline{\D}^2)^+_\Hex, \nabla \in C(\DC)^+_\Hex$ and $\delta \in \C_\Hex^+$ such that for all $z, w \in \Hex$,
			\[
			1-f(z)\overline{f(w)}=\xi(z, w)(1-E(z)\overline{E(w)})+\nabla(z, w)(1-e(z)\overline{e(w)})+(1-z^{(3)}\overline{w}^{(3)})\delta(z, w),
			\]
			where $E(z)$ and $e(z)$ are as in \eqref{eqn_E(z)};
			\item[$(4)$] $f \in UC(\Hex)$.
		\end{enumerate}
	\end{thm}
	
	\begin{proof}
		We have divided the proof into several steps for a clear understanding of the readers.
		
		\medskip 
		
		\noindent $(1) \implies (2)$. Let $k \in AK(\Hex)$. We first prove this implication for functions in $SA(\Hex) \cap \text{Hol}(\overline{\Hex})$. We then show that every function in $SA(\Hex)$ can be approximated by a sequence of functions from $SA(\Hex) \cap \text{Hol}(\overline{\Hex})$, which gives the desired conclusion. Let $f \in SA(\Hex) \cap \text{Hol}(\overline{\Hex})$. Let $M_z=\left(M_{z^{(1)}}, M_{z^{(2)}}, M_{z^{(3)}}, M_{z^{(4)}}\right)$ be the quadruple of operators of multiplication by the coordinate functions $z^{(j)}$ for $1 \leq j \leq 4$ on the reproducing kernel Hilbert space $\HS(k)$ determined by the kernel $k$. It is not difficult to see that each $M_{z^{(j)}}$ is bounded, and
		\begin{equation}\label{eqn_RH_001}
			\|M_{z^{(3)}}\|\leq 1, \quad \|\Psi(\alpha, M_{z^{(2)}}, M_{z^{(3)}}, M_{z^{(4)}})\| \leq 1 \quad \text{and} \quad \|\psi_{\alpha_1, \alpha_2}(M_{z^{(1)}}, M_{z^{(2)}}, M_{z^{(3)}}, M_{z^{(4)}})\| \leq 1
		\end{equation}
		for all $\alpha, \alpha_1, \alpha_2 \in \DC$. Let $\{z_1, \dotsc, z_m\} \subseteq \Hex$ and let $\HS_m(k)$ be the finite dimensional Hilbert space spanned by 
		$\{k(., z_1), \dotsc, k(., z_m)\}$. Since $M_{z^{(j)}}^*k(., z_\ell)=\overline{z}_\ell^{(j)}k(., z_\ell)$ for $1\leq \ell \leq m$ and $1 \leq j \leq 4$, it follows that $\HS_m(k)$ is jointly invariant under $\left(M_{z^{(1)}}^*, M_{z^{(2)}}^*, M_{z^{(3)}}^*, M_{z^{(4)}}^*\right)$. Let us define the quadruple $\underline{T}=(T_1, T_2, T_3, T_4)$ on $\HS_m(k)$ with each $T_j^*=M_{z^{(j)}}^*|_{\HS_m(k)}$. We have by \eqref{eqn_RH_001} that
		\[
		\|T_3^*\|\leq 1, \quad \|\Psi(\alpha, T_2^*, T_3^*, T_4^*)\| \leq 1 \quad \text{and} \quad \|\psi_{\alpha_1, \alpha_2}(T_1^*, T_2^*, T_3^*, T_4^*)\| \leq 1
		\]
		for all $\alpha, \alpha_1, \alpha_2 \in \DC$. Since $\Hex$ is $(1, 1, 1, 2)$-quasi-balanced, it is not difficult to see that 
		\[
		\|rT_3^*\|< 1, \quad \|\Psi(\alpha, rT_2^*, rT_3^*, r^2T_4^*)\| < 1 \quad \text{and} \quad \|\psi_{\alpha_1, \alpha_2}(rT_1^*, rT_2^*, rT_3^*, r^2T_4^*)\| < 1
		\]
		for all $\alpha, \alpha_1, \alpha_2 \in \DC$ and $0 <r<1$. Consequently, $r\cdot \underline{T}^*=(rT_1^*, rT_2^*, rT_3^*, r^2T_4^*) \in \mathfrak{M}_\Hex$, where $\underline{T}^*=(T_1^*, T_2^*, T_3^*, T_4^*)$. Consider the function given by $\widehat{f}(z)=\overline{f(\bar{z})}$, which is holomorphic on $\overline{\Hex}$ since $f \in \text{Hol}(\overline{\Hex})$. Since $f \in SA(\Hex)$, we have that
		\[
		\|\widehat{f}(S_1 ,S_2, S_3, S_4)\|=\|f(S_1^*, S_2^*, S_3^*, S_4^*)^*\|=\|f(S_1^*, S_2^*, S_3^*, S_4^*)\| \leq 1
		\]
		for every $(S_1, S_2, S_3, S_4) \in \mathfrak{M}_{\Hex}$ and so, $\widehat{f} \in SA(\Hex) \cap \text{Hol}(\overline{\Hex})$. Since $\overline{\Hex}$ is polynomially convex (see Theorem 6.7 in \cite{Biswas}), Oka-Weil theorem ensures a sequence of polynomials $\{p_n\}_{n \in \N}$ such that $\underset{n \to \infty}{\lim}\|p_n-f\|_{\infty, \overline{\Hex}}=0$. Evidently,  $\{\widehat{p}_n\}_{n \in \N}$ converges uniformly over $\overline{\Hex}$ to $\widehat{f}$. For each $z_j=(z_j^{(1)}, z_j^{(2)}, z_j^{(3)}, z_j^{(4)})$ and $r \in (0, 1)$, let $r.z_j=(rz_j^{(1)}, rz_j^{(2)}, rz_j^{(3)}, r^2z_j^{(4)})$. Take $c_1, \dotsc, c_m \in \C$. Then 
		\begin{small} 
			\begin{align*}
				\left\|\overset{m}{\underset{i, j=1}{\sum}}c_j\overline{f(r\cdot z_j)}k(., z_j)\right\|_{\HS_m(k)}^2 
				=\lim_{n \to \infty}\left\|\widehat{p_n}(r.\underline{T}^*)\left(\overset{m}{\underset{i, j=1}{\sum}}c_jk(., z_j)\right)\right\|_{\HS_m(k)}^2
				&=\left\|\widehat{f}(r.\underline{T}^*)\left(\overset{m}{\underset{i, j=1}{\sum}}c_jk(., z_j)\right)\right\|^2_{\HS_m(k)} \\
				& \leq \left\|\overset{m}{\underset{i, j=1}{\sum}}c_jk(., z_j)\right\|^2_{\HS_m(k)}, 
			\end{align*}
		\end{small}
		where the last inequality holds, because $\widehat{f} \in SA(\Hex)$ and $r\cdot \underline{T}^* \in \mathfrak{M}_\Hex$ for $0<r<1$. Letting $r \uparrow 1$ in the above inequality, it follows that 
		$\displaystyle 
		\overset{m}{\underset{i, j=1}{\sum}}\overline{c}_ic_j\left(1-f(z_i)\overline{f(z_j)}\right)k(z_i, z_j) \geq 0.
		$
		Thus, $(1-f(z)\overline{f(w)})k(z, w) \succcurlyeq 0$ for all $k \in AK(\Hex)$ and $f \in SA(\Hex) \cap \text{Hol}(\overline{\Hex})$. Now, let us consider $f \in SA(\Hex)$ and let $0 <r<1$. Following the proof of Proposition 6.9 in \cite{Biswas}, it follows that $r\cdot z \in \Hex$ for all $z \in \overline{\Hex}$. Consider the map $f_r: \overline{\Hex} \to \C$ defined as $f_r(z^{(1)}, z^{(2)}, z^{(3)}, z^{(4)})=f(rz^{(1)}, rz^{(2)}, rz^{(3)}, r^2z^{(4)})$. Evidently, the map $f_r \in \text{Hol}(\overline{\Hex})$. Furthermore, $f_r \in SA(\Hex)$ since $f \in SA(\Hex)$ and $\Hex$ is a $(1, 1, 1, 2)$-quasi-balanced domain. Therefore, $(1-f_r(z)\overline{f_r(w)})k(z, w) \succcurlyeq 0$ for all $k \in AK(\Hex)$ and $0<r<1$. Letting $r\uparrow 1$ gives the desired conclusion.
		
		\medskip 
		
		\noindent $(2) \implies (3)$. Suppose $(1-f(z)\overline{f(w)})k(z, w) \succcurlyeq 0$ for all $k \in AK(\Hex)$. Consider the map $g: \Hex \times \Hex \to \C$ defined as $g(z, w)=1-f(z)\overline{f(w)}$.	Evidently, $g$ is a self-adjoint function such that $gk \succcurlyeq 0$ for all $k \in AK(\Hex)$. The desired conclusion now follows from Theorem \ref{thm_sa_H}.
		
		\medskip 
		
		\noindent $(3) \implies (4)$. Suppose there exist $\xi \in C(\overline{\D}^2)^+_\Hex, \nabla \in C(\DC)^+_\Hex$ and $\delta \in \C_\Hex^+$ such that for all $z, w \in \Hex$,
		\begin{align}\label{eqn_RH_002}
			1-f(z)\overline{f(w)}=\xi(z, w)(1-E(z)\overline{E(w)})+\nabla(z, w)(1-e(z)\overline{e(w)})+(1-z^{(3)}\overline{w}^{(3)})\delta(z, w).
		\end{align}
		By Proposition \ref{prop_prelim_I_H}, there exist Hilbert spaces $\HS_1, \HS_2$ and functions $L_1: \Hex \to \mathcal{B}(C(\DC^2),\HS_1)$ and $L_2: \Hex \to \mathcal{B}(C(\DC), \HS_2)$ satisfying 
		\[
		\xi(z, w)(f\overline{h})=\langle L_1(z)f, L_1(w)h \rangle_{\HS_1}\quad  \text{and} \quad \nabla(z, w)(u\overline{v})=\langle L_2(z)u, L_2(w)v \rangle_{\HS_2}
		\]
		for all $(z, w) \in \Hex \times \Hex, f, h \in C(\DC^2)$ and $u, v  \in C(\DC)$. Also, there exist unital $*$-representations $\rho_1: C(\DC^2) \to \mathcal{B}(\HS_1)$ and $\rho_2: C(\DC) \to \mathcal{B}(\HS_2)$ such that 
		\begin{align}\label{eqn_RH_003}
			L_1(z)(fh)=\rho_1(f)L_1(z)(h) \quad \text{and} \quad L_2(z)(uv)=\rho_2(u)L_2(z)(v)
		\end{align}
		for all $z \in \Hex, \{f, h\} \subseteq C(\DC^2)$ and $\{u, v\} \subseteq  C(\DC)$. Since $\delta$ is a weak kernel on $\Hex \times \Hex$, we have by Theorem 2.53 in \cite{Agler_McCarthy} that there exist a Hilbert space $\HS_3$ and a function $g: \Hex \to \HS_3$ such that $\delta(z, w)=\langle g(z), g(w) \rangle_{\HS_3}$ for all $(z, w) \in \Hex \times \Hex$. Hence, \eqref{eqn_RH_002} can be re-written as
		\begin{align*}
			1-f(z)\overline{f(w)}
			&=\langle L_1(z)1, L_1(w)1 \rangle_{\HS_1}-\langle L_1(z)E(z), L_1(w)E(w) \rangle_{\HS_1}+\langle L_2(z)1, L_2(w)1 \rangle_{\HS_2}\\
			& \quad -\langle L_2(z)e(z), L_2(w)e(w) \rangle_{\HS_2} +\langle g(z), g(w) \rangle_{\HS_3}-\langle z^{(3)}g(z), w^{(3)}g(w) \rangle_{\HS_3}
		\end{align*}
		for all $z, w \in \Hex$. Let $\HS=\HS_1\oplus \HS_2 \oplus \HS_3$. We have by \eqref{eqn_RH_003} for all $z, w \in \Hex$ that
		\begin{align}\label{eqn_RH_004}
			& 1+\left\langle X(z) \begin{bmatrix} L_1(z)1 \\ L_2(z)1 \\ g(z)\end{bmatrix}, X(w)\begin{bmatrix} L_1(w)1 \\ L_2(w)1 \\ g(w)\end{bmatrix} \right\rangle_{\HS} =f(z)\overline{f(w)}+\left\langle  \begin{bmatrix} L_1(z)1 \\ L_2(z)1 \\ g(z)\end{bmatrix}, \begin{bmatrix} L_1(w)1 \\ L_2(w)1 \\ g(w)\end{bmatrix} \right\rangle_{\HS}, 
		\end{align}
		where $X(z)=\begin{bmatrix} \rho_1(E(z)) & 0 & 0\\ 
			0 & \rho_2(e(z)) & 0\\
			0 & 0 & z^{(3)}I_{\HS_3}
		\end{bmatrix}$. Consider the subspaces of $\C \oplus \HS$ given by
		\begin{align*}
			\HS^{(1)}
			= \overline{\text{span}} \left\{
			\begin{bmatrix}
				1 \\
				X(z)
				\begin{bmatrix}
					L_1(z)1 \\
					L_2(z)1 \\
					g(z)
				\end{bmatrix}
			\end{bmatrix}
			: z \in \Hex
			\right\} \quad \text{and} \quad \HS^{(2)}
			= \overline{\text{span}} \left\{
			\begin{bmatrix}
				f(z) \\
				\begin{bmatrix}
					L_1(z)1 \\
					L_2(z)1 \\
					g(z)
				\end{bmatrix}
			\end{bmatrix}
			: z \in \Hex
			\right\}.
		\end{align*}
		It follows from \eqref{eqn_RH_004} that the linear operator $V: \HS^{(1)} \to \HS^{(2)}$ defined as 
		\[
		V\begin{bmatrix}
			1 \\
			X(z)
			\begin{bmatrix}
				L_1(z)1 \\
				L_2(z)1 \\
				g(z)
			\end{bmatrix}
		\end{bmatrix}
		=\begin{bmatrix}
			f(z) \\
			\begin{bmatrix}
				L_1(z)1 \\
				L_2(z)1 \\
				g(z)
			\end{bmatrix}
		\end{bmatrix}
		\]
		is an isometry. Adding an infinite-dimensional summand to $\HS^{(1)}$, if necessary, $V$ can be extended to a unitary from $\C \oplus \HS$ onto $\C \oplus \HS$ (see Section 11.3 in \cite{Agler_McCarthy}). With respect to the decomposition $\C \oplus \HS$, we can write $V=\begin{bmatrix} A & B \\ C & D \end{bmatrix}$. By definition of $V$, it then follows that 
		\begin{align}\label{eqn_RH_005}
			f(z)=A1+BX(z)\begin{bmatrix}
				L_1(z)1 \\
				L_2(z)1 \\
				g(z)
			\end{bmatrix} \quad \text{and} \quad C1+DX(z)\begin{bmatrix}
				L_1(z)1 \\
				L_2(z)1 \\
				g(z)
			\end{bmatrix}=\begin{bmatrix}
				L_1(z)1 \\
				L_2(z)1 \\
				g(z)
			\end{bmatrix}.
		\end{align}
		Since $\|X(z)\|<1$ and $\|D\| \leq 1$, it follows from \eqref{eqn_RH_005} that $\begin{bmatrix}
			L_1(z)1 \\
			L_2(z)1 \\
			g(z)
		\end{bmatrix}=(I_\HS-DX(z))^{-1}C1$ and so, $f(z)=A+BX(z)(I_\HS-DX(z))^{-1}C$. Hence, $f \in UC(\Hex)$.  
		
		\medskip 
		
		\noindent $(4) \implies (1)$. The proof follows the same line of argument as that of Proposition 3.2 in \cite{Drit2007_I}. We have already established this part for functions in $UC(\Hex)$ associated with simple representations in Lemma \ref{lem_prelim_III_H}. We now show that an arbitrary function $f \in UC(\Hex)$ can be approximated pointwise by a net $\{f_\beta\} \subseteq UC(\Hex)$ consisting of functions associated with simple representations, from which the desired conclusion follows. From here onwards, the proof is divided into further steps for the better understanding of the reader.
		
		\smallskip 
		
		\noindent{\textit{Step (1).}} Let $f \in UC(\Hex)$. Then there exist associated unital $*$-representations $\rho_1: C(\DC^2) \to \mathcal{B}(\HS_1)$ and $\rho_2: C(\DC) \to \mathcal{B}(\HS_2)$ together with a Hilbert space $\HS_3$ such that 
		\[
		f(z)=A+BX(z)(I_{\HS}-DX(z))^{-1}C, \quad \text{where} \quad X(z)=\begin{bmatrix}
			\rho_1(E(z)) & 0 & 0 \\
			0 & \rho_2(e(z)) & 0 \\
			0 & 0 & z^{(3)}I_{\HS_3}
		\end{bmatrix}
		\]
		and $\HS=\HS_1\oplus \HS_2 \oplus \HS_3$, and $V=\begin{bmatrix} A & B \\ C & D\end{bmatrix}$ is a unitary on $\C \oplus \HS$. By representation theorem for spectral measures, there exist unique $\mathcal{B}(\HS_1)$-valued and $\mathcal{B}(\HS_2)$-valued spectral measures $\mu_1$ and $\mu_2$ on the Borel $\sigma$-algebra of $\DC^2$ and $\DC$, respectively, such that
		\begin{align}\label{eqn_RH_006}
			\rho_1(h)=\int_{\DC^2}h(\al^{(1)}, \al^{(2)})d\mu_1 \quad \text{and} \quad \rho_2(g)=\int_{\DC}g(\al)d\mu_2
		\end{align}
		for all $h \in C(\DC^2)$ and $g \in C(\DC)$.
		\smallskip 
		
		\noindent{\textit{Step (2).}}  In this step, using the above representations of $\rho_1$ and $\rho_2$, we construct a net with the desired properties. Consider a collection $\mathfrak{F}$ consisting of pairs $\beta=(F, \epsilon)$, where $F$ is a finite subset of $\Hex$ and $\epsilon >0$ ordered by $(F_1, \epsilon_1)\leq (F_2, \epsilon_2)$ if $F_1 \subseteq F_2$ and $\epsilon_1 \geq \epsilon_2$. This makes $\mathfrak{F}$ a directed set. Let $\beta=(F, \epsilon) \in \mathfrak{F}$ and let us consider the collection given by
		\[
		\Lambda=\left\{\psi_{\alpha_1, \alpha_2}, \chi_\alpha, \pi_3 : (\alpha_1, \alpha_2) \in \DC^2, \alpha \in \DC \right\},
		\]
		where $\chi_\alpha, \pi_3: \Hex \to \C$ are given by 
		\[
		\chi_{\alpha}(z^{(1)}, z^{(2)}, z^{(3)}, z^{(4)})=\Psi(\alpha, z^{(2)}, z^{(3)}, z^{(4)}) \quad \text{and} \quad \pi_3(z^{(1)}, z^{(2)}, z^{(3)}, z^{(4)})=z^{(3)}.
		\]
		Clearly, $\Lambda \subseteq B(\Hex, \DC)$ (or equivalently, $\DC^{\Hex})$, the collection of bounded functions from $\Hex$ into $\DC$, which is endowed with the topology of pointwise convergence. By Tychonov's theorem, $B(\Hex, \DC)$ is a compact Hausdorff space. Now $\DC$ is a Tychonov space in the usual metric topology (that is, points are closed and for any closed set and point
		disjoint from it, there is a continuous bounded function separating the two). Consequently, $\DC^\Hex$ is Tychonov and so, $\Lambda$ is a Tychonov space. Let $\Lambda_1=\left\{\psi_{\alpha_1, \alpha_2}: (\alpha_1, \alpha_2) \in \DC^2 \right\}$ and $\Lambda_2=\left\{ \chi_\alpha : \alpha \in \DC \right\}$.  Consider the maps $\eta_1: \DC^2 \to \Lambda, \eta_2: \DC \to \Lambda$ given by
		$\eta_1(\alpha_1, \alpha_2)=\psi_{\alpha_1, \alpha_2}$ and $\eta_2(\alpha)=\chi_\alpha$. Evidently, $\eta_1, \eta_2$ are continuous maps with $\eta_1(\DC^2)=\Lambda_1$ and $\eta_2(\DC)=\Lambda_2$. Therefore, $\Lambda_1, \Lambda_2$ are compact subsets of $\Lambda$. By compactness of $\Lambda_1$, there exists a finite collection $\mathcal{U}=\{U^{\beta}_1, \dots, U^{\beta}_m\}$ of nonempty open sets in $\Lambda$ which covers $\{\psi_{\alpha_1, \alpha_2} : (\alpha_1, \alpha_2) \in \mathbb{D}^2\}$ with the property that $|\psi'(z) - \psi''(z)| < \varepsilon$ for every $\psi', \psi'' \in U^{\beta}_j$ for $j = 1, \dots, m$ and $z \in F$. We construct a partition $\{\Delta_1^\beta, \dotsc, \Delta_m^\beta\}$ of $\Lambda_1$ from $\mathcal{U}$ as follows:
		\[
		\Delta^{\beta}_1 = U^{\beta}_1, \quad
		\Delta^{\beta}_2 = U^{\beta}_2 \setminus U^{\beta}_1, \quad 
		\dots, \quad 
		\Delta^{\beta}_m = U^{\beta}_m \setminus (U_1^{\beta} \cup \dotsc \cup U_{m-1}^\beta).
		\]
		Evidently, $\{\Delta^{\beta}_1, \dotsc, \Delta^{\beta}_m\}$ consists of mutually disjoint Borel subsets of $\Lambda$ that cover $\Lambda_2$ with the property that for $1 \leq j \leq m$, we have 
		\begin{align}\label{eqn_RH_007}
			|\psi'(z)-\psi''(z)|<\epsilon \quad \text{for every} \quad \psi', \psi'' \in \Delta_j^\beta, \ z \in F.
		\end{align}
		For the compact set $\Lambda_1$, we proceed analogously to the preceding construction and obtain a partition $\{\Omega_1^\beta, \dotsc, \Omega_\ell^\beta\}$ of $\Lambda_1$ into Borel subsets with the following property: for each $1 \leq i \leq \ell$,
		\begin{align}\label{eqn_RH_008}
			|\chi'(z) - \chi''(z)| < \epsilon 
			\quad \text{for all } \chi', \chi'' \in \Omega_i^\beta 
			\text{ and for every } z \in F.
		\end{align}
		Also, consider collections $\{\psi_1^\beta, \dotsc, \psi_m^\beta\}$ and $\{\chi_1^\beta, \dotsc, \chi_\ell^\beta\}$ such that $\psi_j^\beta=\psi_{\alpha_j^{(1)}, \alpha_j^{(2)}} \in \Delta_j^\beta$ for some $(\alpha_j^{(1)}, \alpha_j^{(2)}) \in \DC^2$ and $\chi_i=\chi_{\alpha_i} \in \Omega_i^\beta$ for some $\alpha_i \in \DC$, where $j=1, \dotsc, m$ and $i=1, \dotsc, \ell$. Define the maps $\rho_{1, \beta}: C(\DC^2) \to \mathcal{B}(\HS_1)$ and $\rho_{2, \beta}: C(\DC) \to \mathcal{B}(\HS_2)$ as follows:
		\[
		\rho_{1, \beta}(h)=\overset{m}{\underset{j=1}{\sum}}\mu_1\left(\eta_1^{-1}(\Delta_j^\beta)\right)h(\al_j^{(1)}, \al_j^{(2)}) \quad \text{and} \quad \rho_{2, \beta}(g)=\overset{\ell}{\underset{i=1}{\sum}}\mu_2\left(\eta_2^{-1}(\Omega_i^\beta)\right)g(\al_i).
		\]
		Since $\{\Delta_1^\beta, \dotsc, \Delta_m^\beta\}$ is a partition of $\Lambda_1$ and $\eta_1(\DC^2)=\Lambda_1$, it follows that $\{\eta_1^{-1}(\Delta_1^\beta), \dotsc, \eta_1^{-1}(\Delta_m^\beta)\}$ is a partition of $\DC^2$. Consequently, the operators $\mu_1\left(\eta_1^{-1}(\Delta_j^\beta)\right)$ are pairwise orthogonal projections for $1 \leq j \leq m$ such that $\overset{m}{\underset{j=1}{\sum}}\mu_1\left(\eta_1^{-1}(\Delta_j^\beta)\right)=I_{\HS_1}$ and so, $\rho_{1, \beta}$ is a simple representation.  Similarly, one can show that $\rho_{2, \beta}$ is a simple representation. It follows from \eqref{eqn_RH_007} and \eqref{eqn_RH_008} that 
		\[
		\|\rho_{1, \beta}(E(z))-\rho_1(E(z))\| \leq \epsilon \quad \text{and} \quad  \|\rho_{2, \beta}(e(z))-\rho_2(e(z))\|\leq \epsilon
		\]
		for every $z \in F$. Consider the function $f_\beta: \Hex \to \C$ given by 
		\[
		f_\beta(z)=A+BX_\beta(z)(I_\HS-DX_\beta(z))^{-1}C, \quad \text{where} \quad X_\beta(z)=\begin{bmatrix}
			\rho_{1, \beta}(E(z)) & 0 & 0 \\
			0 & \rho_{2,\beta}(e(z)) & 0 \\
			0 & 0 & z^{(3)}I_{\HS_3}
		\end{bmatrix}.
		\]
		Putting everything together, we have constructed a net $\{f_\beta:\beta \in \mathfrak{F}\}$ of functions in $UC(\Hex)$ associated with simple representations. We have by Lemma \ref{lem_prelim_III_H} that each of these $f_\beta \in SA(\Hex)$.
		
		\smallskip
		
		\noindent{\textit{Step (3).}} Here, we prove that the net $\{f_\beta\}$ converges pointwise to $f$ on $\Hex$. Let $z \in \Hex$ and $\epsilon'>0$. Choose a finite set $F$ in $\Hex$ such that $z \in F$ and consider $\beta'=(F, \epsilon') \in \mathfrak{F}$. Since $\|E(w)\|_{\infty, \DC^2}<1, \|e(w)\|_{\infty, \DC}<1$ for all $w \in \Hex$, we can define positive scalars $\delta, r_1, r_2$ and $\epsilon$ as follows:
		\[
		\delta=\underset{w \in F}{\min}\left\{\frac{1-\|E(w)\|_{\infty, \DC^2}}{2}, \frac{1-\|e(w)\|_{\infty, \DC}}{2}\right\}, \ \ r_1=1-\frac{\delta}{2}, \ \ r_2=\underset{w \in F}{\max}|w^{(3)}| \ \ \text{and} \ \ \epsilon<\min\{\delta\slash 2, \epsilon'\}.
		\]
		Choose $\beta=(F, \epsilon) \in \mathfrak{F}$. By the definition of $\mathfrak{F}$, it follows that $\beta' \leq \beta$. Since $\rho_1, \rho_2$ are unital $*$-representations, we have by Step (2) that for each $w \in F$, $\|X_\beta(w)-X(w)\| \leq \epsilon$ together with
		\begin{align*}
			& \|\rho_{1, \beta}(E(w))\| \leq \|\rho_{1, \beta}(E(w))-\rho_1(E(w))\|+\|\rho_1(E(w))\| \leq \epsilon +1-2\delta<r_1, \ \text{and} \\
			& \|\rho_{2, \beta}(e(w))\| \leq \|\rho_{2, \beta}(e(w))-\rho_2(e(w))\|+\|\rho_2(e(w))\| \leq \epsilon +1-2\delta<r_1.
		\end{align*}
		Consequently, $\|X_\beta(w)\| =\max\{\|\rho_{1, \beta}(E(w))\|, \|\rho_{2, \beta}(e(w)), |w^{(3)}|: w\in F \} \leq r$ for all $w \in F$, where $r=\max\{r_1, r_2\}$. Since $D$ is a contraction, we have that $\|DX_\beta{(w)}\| \leq r$ for all $w \in F$. Let $w \in F$. Note that 
		\begin{small}
			\begin{align*}
				(DX_\beta(w))^n-(DX(w))^n&=D(X_\beta(w)-X(w))(DX_\beta(w))^{n-1}+(DX(w))(X_\beta(w)-X(w))(DX_\beta(w))^{n-2}\\
				&\quad +\dotsc+(DX(w))^{n-1}D(X_\beta(w)-X(w)) 
			\end{align*}
		\end{small}
		and so, $\|(DX_\beta(w))^n-(DX(w))^n\| \leq n\epsilon r^{n-1}$. For all $w \in F$, we have that
		\begin{align*}
			& \|X_\beta(w)(I_\HS-DX_\beta(w))^{-1}-X(w)(I_\HS-DX(w))^{-1}\|\\
			& \leq \|X_\beta(w)-X(w)\| \|(I_\HS-DX_\beta(w))^{-1}\|+\|X(w)\| \|(I_\HS-DX_\beta(w))^{-1}-(I_\HS-DX(w))^{-1}\|\\
			& \leq \frac{\epsilon}{(1-r)^2}.
		\end{align*}
		It is now easy to see that the bounded net $\{f_\beta\}$ converges pointwise to $f$ on $\Hex$. 
		
		\smallskip
		
		\noindent \textit{Step (4).} Since the net $\{f_\beta\}$ is uniformly bounded by $1$, we have by Theorem 1.4.31 in \cite{Scheidemann} that there exists a subsequence $\{f_{\beta_k}\}$ of the net that converges uniformly to $f$. Let $\underline{T} \in \mathfrak{M}_\Hex$, and let $T$ be acting on a Hilbert space $\mathcal{K}$. As proved in Step (2), $f_{\beta_k} \in SA(\Hex)$ and so, $\|f_{\beta_k}(\underline{T})\| \leq 1$ for each $k$. By functional calculus representation by Vasilescu \cite{Vasilescu}, it follows that 
		\[
		f_{\beta_k}(\underline{T})=\frac{1}{(2\pi i)^4}\int_{\partial \Omega} M_{\underline{T}}(z)f_{\beta_k}(z)dz,
		\]
		where $\Omega$ is an open set in $\C^4$ containing $\sigma_T(\underline{T})$ with $C^1$-boundary, $\overline{\Omega} \subseteq \Hex$ and $M_{\underline{T}}(z)$ is the Martinelli kernel corresponding to the quadruple $\underline{T}$. For each $x, y \in \mathcal{K}$, consider $d\mu(z)=\langle M_{\underline{T}}(z)x, y\rangle dz$, which gives a measure on the Borel subsets of $\partial \Omega$. By dominated convergence theorem, we have
		\[
		\lim_{k \to \infty}\langle f_{\beta_k}(\underline{T})x, y\rangle_{\mathcal{K}} = \frac{1}{(2\pi i)^4}\int_{\partial \Omega} \lim_{k \to \infty}f_{\beta_k}(z)d\mu(z)=\langle f(\underline{T})x, y\rangle_{\mathcal{K}} 
		\]
		and so, $\|f(\underline{T})\| \leq 1$. Thus, $f \in SA(\Hex)$, which completes the proof. 
	\end{proof}
	
	\subsection*{Interpolation theorem on the hexablock.} Given $\{z_1, \dotsc, z_n\} \subseteq \Hex$ and $\lm_1, \dotsc, \lm_n \in \DC$, the interpolation problem associated with $SA(\Hex)$ asks for necessary and sufficient conditions for the existence of a function $f \in SA(\Hex)$ that maps each $z_i$ to $\lm_i$. As an application of Theorem \ref{thm_realization_H}, various characterizations of such an interpolation problem for $\Hex$ are presented below.
	
	\begin{thm}\label{thm_interpolation_H}
		Let $F=\{z_1, \dotsc, z_n\} \subseteq \Hex$ and $\lm_1, \dotsc, \lm_n \in \DC$. Then the following are equivalent:
		\begin{enumerate}[leftmargin=*]
			\item[$(1)$] there exists a function $f \in SA(\Hex)$ such that $f(z_i)=\lm_i$ for $1 \leq i \leq n$;
			\item[$(2)$] $\begin{bmatrix} (1-\lm_i\overline{\lm_j})k(z_i, z_j)\end{bmatrix}_{i, j=1}^n \geq 0$ for all $k \in AK(\Hex)$;
			\item[$(3)$] there exist $\xi \in C(\overline{\D}^2)^+_F, \nabla \in C(\DC)^+_F$ and $\delta \in \C_F^+$ such that for all $i, j \in \{1, \dotsc, n\}$,
			\[
			1-\lambda_i\overline{\lambda}_j=\xi(z_i, z_j)(1-E(z_i)\overline{E(z_j)})+\nabla(z_i, z_j)(1-e(z_i)\overline{e(z_j)})+(1-z_i^{(3)}\overline{z}_j^{(3)})\delta(z_i, z_j).
			\] 
		\end{enumerate}
	\end{thm}
	
	\begin{proof}
		The implication $(1) \implies (2)$ follows from Theorem \ref{thm_realization_H}. 
		The implication $(2) \implies (3)$ follows from Theorem \ref{prop_prelim_I_H} 
		by applying it to the self-adjoint map $g : F \times F \to \mathbb{C}$ given by 
		$g(z_i, z_j) = 1 - \lm_i \overline{\lm}_j$. It remains to prove that $(3) \implies (1)$. Suppose there exist $\xi \in C(\overline{\D}^2)^+_F, \nabla \in C(\DC)^+_F$ and $\delta \in \C_F^+$ such that for all $i, j=1, \dotsc, n$,
		\begin{align}\label{eqn_IP_002}
			1-\lambda_i\overline{\lambda}_j=\xi(z_i, z_j)(1-E(z_i)\overline{E(z_j)})+\nabla(z_i, z_j)(1-e(z_i)\overline{e(z_j)})+(1-z_i^{(3)}\overline{z}_j^{(3)})\delta(z_i, z_j).
		\end{align} 
		By Proposition \ref{prop_prelim_I_H}, there exist Hilbert spaces $\HS_1, \HS_2$ and functions $L_1: F \to \mathcal{B}(C(\DC^2),\HS_1)$ and $L_2: F \to \mathcal{B}(C(\DC), \HS_2)$ satisfying 
		\[
		\xi(z_i, z_j)(f\overline{h})=\langle L_1(z_i)f, L_1(z_j)h \rangle_{\HS_1}\quad  \text{and} \quad \nabla(z_i, z_j)(u\overline{v})=\langle L_2(z_i)u, L_2(z_j)v \rangle_{\HS_2}
		\]
		for $1 \leq i, j \leq n, \{f, h\} \subseteq C(\DC^2)$ and $\{u, v\}  \subseteq C(\DC)$. Also, there exist unital $*$-representations $\rho_1: C(\DC^2) \to \mathcal{B}(\HS_1)$ and $\rho_2: C(\DC) \to \mathcal{B}(\HS_2)$ such that 
		\begin{align}\label{eqn_IP_003}
			L_1(z_i)(fh)=\rho_1(f)L_1(z_i)(h) \quad \text{and} \quad L_2(z_i)(uv)=\rho_2(u)L_2(z_i)(v)
		\end{align}
		for $1 \leq i, j \leq n, \{f, h\} \subseteq C(\DC^2)$ and $\{u, v\} \subseteq  C(\DC)$. Since $\delta$ is a weak kernel on $F \times F$, we have by Theorem 2.53 in \cite{Agler_McCarthy} that there exists a Hilbert space $\HS_3$ and a function $g: F \to \HS_3$ such that $\delta(z_i, z_j)=\langle g(z_i), g(z_j) \rangle_{\HS_3}$ for $1 \leq i, j \leq n$. Hence, \eqref{eqn_IP_002} can be re-written as
		\begin{align*}
			1-\lm_i \overline{\lm}_j
			&=\langle L_1(z_i)1, L_1(z_j)1 \rangle_{\HS_1}-\langle L_1(z_i)E(z_i), L_1(z_j)E(z_j) \rangle_{\HS_1}+\langle L_2(z_i)1, L_2(z_j)1 \rangle_{\HS_2} \notag \\
			& \quad -\langle L_2(z_i)e(z_i), L_2(z_j)e(z_j) \rangle_{\HS_2} +\langle g(z_i), g(z_j) \rangle_{\HS_3}-\langle z_i^{(3)}g(z_i), z_j^{(3)}g(z_j) \rangle_{\HS_3}
		\end{align*}
		for $1 \leq i, j \leq n$. Let $\HS=\HS_1\oplus \HS_2 \oplus \HS_3$. Consider the subspaces of $\C \oplus \HS$ given by
		\begin{align*}
			\HS^{(1)}
			= \overline{\text{span}} \left\{
			\begin{bmatrix}
				1 \\
				\rho_1(E(z_i))L_1(z_i)1 \\
				\rho_2(e(z_i))L_2(z_i)1 \\
				z_i^{(3)}g(z_i)
			\end{bmatrix}
			: 1 \leq i \leq n
			\right\} \quad \text{and} \quad \HS^{(2)}
			= \overline{\text{span}} \left\{
			\begin{bmatrix}
				\lm_i \\
				L_1(z_i)1 \\
				L_2(z_i)1 \\
				g(z_i)
			\end{bmatrix}
			: 1 \leq i \leq n
			\right\}.
		\end{align*}
		Consequently, the linear operator $V: \HS^{(1)} \to \HS^{(2)}$ defined as 
		\[
		V\begin{bmatrix}
			1 \\
			\rho_1(E(z_i))L_1(z_i)1 \\
			\rho_2(e(z_i))L_2(z_i)1 \\
			z_i^{(3)}g(z_i)
		\end{bmatrix}
		=\begin{bmatrix}
			\lm_i \\
			L_1(z_i)1 \\
			L_2(z_i)1 \\
			g(z_i)
		\end{bmatrix}
		\]
		is an isometry. Now extend $V$ to a unitary on $\C \oplus \HS$. Then one write $V=\begin{bmatrix} A & B \\ C & D \end{bmatrix} : \C \oplus \HS \to \C \oplus \HS$. By Theorem \ref{thm_realization_H}, the map $f: \Hex \to \C$ given by
		\[
		f(z)=A+BX(z)\left(I_\HS-DX(z)\right)^{-1}C, \quad \text{where} \quad X(z)=\begin{bmatrix} \rho_1(E(z)) & 0 & 0\\ 0 & \rho_2(e(z)) & 0 \\ 0 & 0 & z^{(3)} \end{bmatrix}
		\]
		is in $SA(\Hex)$. It is not difficult to see that $f(z_i)=\lambda_i$ for $1 \leq i \leq n$, which completes the proof. 
	\end{proof}

	\subsection*{Extension theorem on the hexablock} 
	Let $W$ be a subset of $\Hex$. Denote by $\mathscr{HE}(W)$ the collection of all bounded functions on $W$ that admits an extension to a holomorphic function in a neighbourhood of $W$. In this section, we study the following extension problem: given a subset $W \subseteq \mathbb{H}$, determine necessary and sufficient conditions under which a function $f \in \mathscr{HE}(W)$ admits a norm-preserving extension to a function $g \in H^\infty(\mathbb{H})$ such that $g \slash \|f\|_{\infty, W}\in SA(\Hex)$. To do so, we introduce some terminologies and definitions. Recall from \eqref{eqn_H} that $z=(z^{(1)}, z^{(2)}, z^{(3)}, z^{(4)}) \in \Hex$ if and only if $|z^{(3)}|<1, |\Psi(\al, z^{(2)}, z^{(3)}, z^{(4)})|<1$ and $|\psi_{\al_1, \al_2}(z^{(1)}, z^{(2)}, z^{(3)}, z^{(4)})|<1$ for all $\al, \al_1, \al_2 \in \DC$, where 
	\[
	\Psi(\alpha, z^{(2)}, z^{(3)}, z^{(4)})=\frac{\alpha z^{(4)}-z^{(2)}}{\alpha z^{(3)}-1} \ \ \text{and} \ \ 	\psi_{\al_1, \al_2}(z^{(1)}, z^{(2)}, z^{(3)}, z^{(4)})=\frac{z^{(1)}\sqrt{(1-|\al_1|^2)(1-|\al_2|^2)}}{1-z^{(2)}\al_1-z^{(3)}\al_2+z^{(4)}\al_1\al_2}.
	\]
	Let $Q\Hex$ be the class of all commuting quadruples $\underline{T}=(T_1, T_2, T_3, T_4)$ of Hilbert space operators with $\sigma_T(\underline{T}) \subseteq \Hex$ such that 
	for all $\al \in \DC$ and $(\al_1, \al_2)\in \DC^2$,
	\[
	\|T_3\| \leq 1, \quad \|\Psi(\al, T_2, T_3, T_4)\| \leq 1 \quad \text{and} \quad \|\psi_{\al_1, \al_2}(T_1, T_2, T_3, T_4)\| \leq 1,
	\]
	where $\Psi(\al, T_2, T_3, T_4)=(\al T_4-T_2)(\al T_3-1)^{-1}$ and
	\[
	\psi_{\al_1, \al_2}(T_1, T_2, T_3, T_4)=\sqrt{(1-|\al_1|^2)(1-|\al_2|^2)}T_1(I-\al_1T_2-\al_2T_3+\al_1\al_2T_4)^{-1}.
	\]
	Following the terminology in \cite{Mittal} for quantized domains in $\C^n$, we refer the class $Q\Hex$ as the \textit{quantum hexablock}. By quantization, we mean the process of replacing scalars with operator variables subject to analogous norm constraints. Recall from Section \ref{sec_hexa} that $\mathfrak{M}_\Hex$ is the set of all commuting quadruples $\underline{T}=(T_1, T_2, T_3, T_4)$ of Hilbert space operators such that 
	for all $\al, \al_1, \al_2 \in \DC$,
	\[
	\|T_3\| < 1, \quad \|\Psi(\al, T_2, T_3, T_4)\| < 1 \quad \text{and} \quad \|\psi_{\al_1, \al_2}(T_1, T_2, T_3, T_4)\| < 1.
	\]
	It was proved in the beginning of Section \ref{sec_hexa} that $\sigma_T(\underline{T}) \subseteq \Hex$ for all $\underline{T} \in \mathfrak{M}_\Hex$. Consequently, it follows that the class $\mathfrak{M}_\Hex$ is contained in the quantum pentablock $Q\Hex$.
	
	\smallskip 
	
	For $W \subseteq \Hex$, we say that a commuting quadruple $\underline{T}=(T_1, T_2, T_3, T_4)$ is \textit{subordinate} to $W$ if $\sigma_T(\underline{T}) \subset W$ and $g(\underline{T})=0$ whenever $g$ is holomorphic in a neighbourhood of $W$ and $g|_W=0$. If $f$ is a function on $W$ admitting a holomorphic extension in a neighbourhood of $W$ and $\underline{T}=(T_1, T_2, T_3, T_4)$ is subordinate to $W$, then define $f(\underline{T})$ by setting $f(\underline{T})=g(\underline{T})$, where $g$ is any holomorphic extension of $f$ in a neighbourhood of $W$. The definition of $f(\underline{T})$ is independent of the choice of the holomorphic extension $g$ of $f$. To see this, let $h$ be any other holomorphic extension of $f$ in some neighbourhood of $W$. Clearly, the function $g-h$ is holomorphic on a neighbourhood of $W$ and $g-h=0$ on $W$. Since $\underline{T}$ is subordinate to $W$, we have that $g(\underline{T})=h(\underline{T})$. Having introduced the required definitions, we first present the statement of the main theorem in this section.
	
	\begin{thm}\label{thm_ext_H}
		Let $W \subseteq \Hex$ and let $f \in \mathscr{HE}(W)$ be non-zero. Then there exists $g \in H^\infty(\Hex)$ such that $\displaystyle \frac{1}{\|f\|_{\infty, W}} g \in SA(\Hex), g|_W=f$ and $\|g\|_{\infty, \Hex}=\|f\|_{\infty, W}$ if and only if 
		\[
		\|f(\underline{T})\| \leq \|f\|_{\infty, W} \quad \text{for every $\underline{T} \in Q\Hex$ subordinate to $W$}.
		\]
	\end{thm}
	To prove Theorem \ref{thm_ext_H}, we adopt an approach based on the interpolation theorem, in the spirit of \cite{Agler_McCarthy_2003}, \cite{Tirtha_Sau} and \cite{Jain} for the bidisc, symmetrized bidisc and tetrablock, respectively. A first step in this direction is the following lemma which establishes the necessary condition of Theorem \ref{thm_ext_H}.
	
	\begin{lem}
		Let $W \subseteq \Hex$ and let $f \in \mathscr{HE}(W)$ be non-zero. If there exists $g \in H^\infty(\Hex)$ such that $\displaystyle \frac{1}{\|f\|_{\infty, W}} g \in SA(\Hex), g|_W=f$ and $\|g\|_{\infty, \Hex}=\|f\|_{\infty, W}$, then $
		\|f(\underline{T})\| \leq \|f\|_{\infty, W}$ for every $\underline{T} \in Q\Hex$ subordinate to $W$.
	\end{lem}	
	
	\begin{proof}
		Let $\underline{T}=(T_1, T_2, T_3, T_4) \in Q\Hex$ be subordinate to $W$, and let $\underline{T}$ be acting on a Hilbert space $\mathcal{K}$.  Set $r\cdot z=(rz^{(1)}, rz^{(2)}, rz^{(3)}, r^2z^{(4)})$ for $z=(z^{(1)}, z^{(2)}, z^{(3)}, z^{(4)}) \in \C^4$ and $0<r<1$.  Let $r_n=1-1\slash n$ for $n \in \N$. Since $\Hex$ is $(1, 1, 1, 2)$-quasi-balanced, it follows that $r_n\cdot z \in \Hex$ for all $z \in \Hex$. Define $g_n: \Hex \to \C$ as $g_n(z)=g(r_n\cdot z)$. Then $\{g_n\}$ is a uniformly bounded sequence of holomorphic functions converging pointwise to $g$. An application of the dominated convergence theorem gives that $g_n(\underline{T})$ converges to $g(\underline{T})$ in the weak operator topology. Since $r_n\cdot \underline{T} \in \mathfrak{M}_\Hex$ and $\frac{1}{\|f\|_{\infty, W}}g \in SA(\Hex)$, we have that
		$
		|\langle g(\underline{T})x, y\rangle |=\lim_{n \to \infty}|\langle g(r_n\cdot\underline{T})x, y\rangle | \leq \lim_{n \to \infty}\|g(r_n \cdot \underline{T})\| \|x\| \| y\| \leq \|f\|_{\infty, W} \|x\| \|y\|.
		$
		Thus, $\|f(\underline{T})\| =\|g(\underline{T})\|\leq \|f\|_{\infty, W}$. 
	\end{proof}	
	
	The proof of the converse to Theorem \ref{thm_ext_H} brings out a subtle relationship between classical extremal problems and operator theoretic techniques. We first reformulate the interpolation theorem for $\Hex$ in a form suited to the proof of the extension theorem. Let $z_1=(z_1^{(1)}, z_1^{(2)}, z_1^{(3)}, z_1^{(4)}), \dotsc, z_n=(z_n^{(1)}, z_n^{(2)}, z_n^{(3)}, z_n^{(4)})$ be distinct points in $\Hex$. We denote this data by $\textbf{z}$. Let $K_\textbf{z}$ be the collection of all $n \times n$ strictly positive definite matrices $[k(i, j)]_{i, j=1}^n$ with $k(i, i)=1$ for $1 \leq i \leq n$ such that
	\begin{equation}\label{eqn_301}
		\left.
		\begin{aligned}
			&  \left[\left(1-z_i^{(3)}\overline{z}_j^{(3)}\right)k(i, j)\right]_{i, j=1}^n \geq 0, \\[5pt]
			&  \left[\left(1-\Psi(\al, z_i^{(2)}, z_i^{(3)}, z_i^{(4)})\overline{\Psi(\al,  z_j^{(2)}, z_j^{(3)}, z_j^{(4)})}\right)k(i, j)\right]_{i, j=1}^n \geq 0 
			\quad \text{for all} \ \al \in \DC, \\[5pt]
			&  \left[\left(1-\psi_{\al_1, \al_2}(z_i)\overline{\psi_{\al_1, \al_2}(z_j)}\right)k(i, j)\right] \geq 0 
			\quad \text{for all} \ (\al_1, \al_2) \in \DC^2.
		\end{aligned}
		\right\}
	\end{equation}
	Evidently, $K_\textbf{z}$ is a subset of $n \times n$ self-adjoint complex matrices. We prove that $K_\textbf{z}$ is compact by showing that it is closed and bounded with respect to the matrix norm. The fact that $K_\textbf{z}$ is bounded follows since $\|k\| \leq \text{tr}(k)=n$ for all $k \in K_\textbf{z}$. Let $\{k_m\}$ be a sequence in $K_\textbf{z}$ that converges to $k$ in the matrix norm. By continuity arguments, it follows that $k$ is positive semi-definite and $k$ satisfies the conditions in \eqref{eqn_301}. It remains to show that $k$ is a strictly positive definite matrix. Assume on the contrary that there exists a non-zero vector $v=(v_1, \dotsc, v_n)^t \in \C^n$ such that $kv=0$. Let $M_1, M_2, M_3$ and $M_4$ be  $n \times n$ diagonal matrices whose $(i, i)$-th entries are $z_i^{(1)}, z_i^{(2)}, z_i^{(3)}$ and $z_i^{(4)}$, respectively. It follows from \eqref{eqn_301} that $k(M_1v)=k(M_2v)=k(M_3v)=k(M_4v)=0$ and so, $k(p(M_1, M_2, M_3, M_4)v)=0$ for any holomorphic polynomial $p$ in four variables. Since $v \ne 0$, there exists some $i$ such that $v_i \ne 0$. Since $z_1, \dotsc, z_n$ are distinct, one can choose a polynomial $p$ such that $p(z_i)=1$ and $p(z_j)=0$ for $ 1 \leq j \leq n$ with $j \ne i$. Then $k(p(M_1, M_2, M_3, M_4)v)$ equals $v_i$ times the $(i, i)$-th column of $k$ and thus $k(i, i)=0$, which gives a contradiction as $k(i, i)=1$. Hence, the collection $K_\textbf{z}$ is compact. We shall use the compactness of $K_\textbf{z}$ to prove our main result. To this end, we present our next result, which is simply a reformulation of the equivalence $(1)$ and $(2)$ in Theorem \ref{thm_interpolation_H}.
	
	\begin{thm}\label{thm_int_H_II}
		Let $z_1,\dotsc, z_n$ be distinct points in $\Hex$, and let $\lm_1, \dotsc, \lm_n \in \DC$. Then there exists $g \in SA(\Hex)$ such that $g(z_i)=\lm_i$ for $1 \leq i \leq n$ if and only if 	$\begin{bmatrix} (1-\lm_i\overline{\lm_j})k(i, j)\end{bmatrix}_{i, j=1}^n \geq 0$ for all $k \in K_\textbf{z}$.
	\end{thm} 
	
	We now introduce an auxiliary class of functions in $H^\infty(\Hex)$. For $W \subseteq \Hex$ and $f \in \mathscr{HE}(W)$, set
	\[
	SA_f(\Hex)=\{g \in H^\infty(\Hex): \|g(\underline{T})\| \leq \|f\|_{\infty, W} \ \text{for all $\underline{T} \in \mathfrak{M}_\Hex$} \},
	\]
	which is clearly nonempty. For $g \in H^\infty(\Hex)$, let us define $\|g\|_{\mathfrak{M}_\Hex}=\sup\{\|g(\underline{T})\|: \underline{T} \in \mathfrak{M}_\Hex \}$. Evidently, $\|g\|_{\infty, \Hex} \leq \|g\|_{\mathfrak{M}_\Hex} \leq \|f\|_{\infty, W}$ for all $g \in SA_f(\Hex)$. For the given data $\textbf{z}=\{z_1, \dotsc, z_n\} \subseteq W$ and $\boldsymbol{\lambda}=(\lambda_1, \dotsc, \lambda_n) \in \C^n$, we denote by 
	\[
	\rho_f(\textbf{z}, \boldsymbol{\lambda})=\inf\left\{\|g\|_{\mathfrak{M}_\Hex}: g \in SA_f(\Hex) \ \text{and} \ g(z_i)=\lambda_i, 1 \leq i \leq n \right\}.
	\]
	It is easy to see that $|\lm_i| \leq \rho_f(\textbf{z}, \boldsymbol{\lambda})$ for $1 \leq i \leq n$. Our next result shows that $\rho_f(\textbf{z}, \boldsymbol{\lambda})$ is attained at some function $g$ in $SA_f(\Hex)$, that is, $g(z_i)=\lm_i$ for $1 \leq i \leq n$ and $\rho_f(\textbf{z}, \boldsymbol{\lambda})=\|g\|_{\mathfrak{M}_\Hex}$. Such a $g$ is called an \textit{extremal function} for the data $\textbf{z}$ and $\boldsymbol{\lambda}$.
	
	\begin{lem}\label{lem_304}
		Let $W$ be a subset of $\Hex$ and let $f \in \mathscr{HE}(W)$ be non-zero. For given $\textbf{z}=\{z_1, \dotsc, z_n\} \subseteq W$ and $\boldsymbol{\lambda}=(\lambda_1, \dotsc, \lambda_n) \in \C^n$, there exists an extremal function in $SA_f(\Hex)$ for the data $\textbf{z}$ and $\boldsymbol{\lambda}$.
	\end{lem}
	
	\begin{proof}
		Let $\{g_m\}$ be a sequence in $SA_f(\Hex)$ such that $g_m(z_i)=\lm_i$ for $1 \leq i \leq n$ and $\rho_f(\textbf{z}, \boldsymbol{\lambda})=\underset{m \to \infty}\lim\|g_m\|_{\mathfrak{M}_\Hex}$. Evidently, $\|g\|_{\infty, \Hex} \leq \|g\|_{\mathfrak{M}_\Hex} \leq \|f\|_{\infty, W}$ for every $g \in SA_f(\Hex)$. Therefore, one can find a subsequence $\{g_{m_k}\}$ of $\{g_m\}$ which converges pointwise to $g \in H^\infty(\Hex)$. For every $\underline{T}$ in $\mathfrak{M}_\Hex$, an application of dominated convergence theorem gives that $g_{m_k}(\underline{T})$ converges to $g(\underline{T})$ in the weak-operator topology. Consequently, $\|g(\underline{T})\| \leq \|f\|_{\infty, W}$ for all $\underline{T} \in \mathfrak{M}_\Hex$ and $g(z_i)=\lm_i$ for $1 \leq i \leq n$. By definition of $\rho_f(\textbf{z}, \boldsymbol{\lambda})$, it follows that $\rho_f(\textbf{z}, \boldsymbol{\lambda}) \leq \|g\|_{\mathfrak{M}_\Hex}$. Let $\underline{T}$ be a commuting quadruple of operators acting on a Hilbert space $\mathcal{K}$, and let $\underline{T} \in \mathfrak{M}_\Hex$. Then 
		for every $x, y \in \mathcal{K}$,
		\[
		|\langle g(\underline{T})x, y\rangle |=\lim_{k \to \infty}|\langle g_{m_k}(\underline{T})x, y\rangle | \leq \lim_{k \to \infty}\|g_{m_k}(\underline{T})\| \|x\| \| y\| \leq \lim_{k \to \infty}\|g_{m_k}\|_{\mathfrak{M}_\Hex}\|x\|\|y\|=\rho_f(\textbf{z}, \boldsymbol{\lambda}) \|x\| \|y\|.
		\]
		Therefore, $\|g(\underline{T})\| \leq \rho_f(\textbf{z}, \boldsymbol{\lambda})$ for all $\underline{T} \in \mathfrak{M}_\Hex$ and so, $\|g\|_{\mathfrak{M}_\Hex} \leq \rho_f(\textbf{z}, \boldsymbol{\lambda})$. The proof is complete.
	\end{proof}
	
	We now prove the following lemma, which is the final ingredient in the proof of Theorem \ref{thm_ext_H}.
	
	\begin{lem}\label{lem_305}
		Let $W$ be a subset of $\Hex$ and let $f \in \mathscr{HE}(W)$ be non-zero. If $g \in SA_f(\Hex)$ is an extremal function for the data $\textbf{z}=\{z_1, \dotsc, z_n\} \subseteq W$ and $\boldsymbol{\lambda}=(\lambda_1, \dotsc, \lambda_n) \in \C^n$, then there exists $\underline{S} \in Q\Hex$ subordinate to $\textbf{z}$ such that $\|g(\underline{S})\|=\rho_f(\textbf{z}, \boldsymbol{\lambda})$.
	\end{lem}

	\begin{proof}
		Suppose $g \in SA_f(\Hex)$ is an extremal function for the data $\textbf{z}=\{z_1, \dotsc, z_n\} \subseteq W$ and $\boldsymbol{\lambda}=(\lambda_1, \dotsc, \lambda_n) \in \C^n$.	For the sake of brevity, let $\rho=\rho_f(\textbf{z}, \boldsymbol{\lambda})$. If $\rho=0$, then $\|g\|_{\infty, \Hex} \leq \|g\|_{\mathfrak{M}_\Hex}=\rho=0$ and so, $g=0$. In this case, one can choose $\underline{S}=(z_1^{(1)}I, z_1^{(2)}I, z_1^{(3)}I, z_1^{(4)}I)$ on any Hilbert space $\mathcal{K}$. Evidently, $\underline{S} \in Q\Hex$ is subordinate to $\textbf{z}$ and $\|g(\underline{S})\|=\rho=0$.	Let us assume that $\rho>0$. Since $\|g(\underline{T})\| \leq \|g\|_{\mathfrak{M}_\Hex} = \rho$ for every $\underline{T} \in \mathfrak{M}_\Hex$, it follows that $\frac{1}{\rho}g \in SA(\Hex)$ and $\frac{1}{\rho}g(z_i)=\lm_i\slash \rho$ for $1\leq i \leq n$. By Theorem \ref{thm_int_H_II}, $\left[(\rho^2-\lm_i\overline{\lm}_j)k(i, j)\right]_{i, j=1}^n \geq 0$ for all $k \in K_\textbf{z}$. Consider the set
		\[
		\Lambda=\left\{\lambda : 0 < \lambda \leq \|f\|_{\infty, W} \ \ \text{and} \ \ \left[(\lambda^2-\lm_i\overline{\lm}_j)k(i, j)\right]_{i, j=1}^n \geq 0 \ \text{for all} \ k \in K_\textbf{z} \right\}.
		\]
		Clearly, $\rho \in \Lambda$. Let $\lambda \in Y$. Note that $\lambda_i \slash \lambda \in \DC$ since $(\lambda^2 - |\lambda_i|^2)k(i,i) \geq 0$ for $k \in K_{\mathbf{z}}$ and $1 \leq i \leq n$. By Theorem \ref{thm_int_H_II}, there exists $h_\lambda \in SA(\Hex)$ such that $h_\lm(z_i)=\lm_i\slash \lm$ for $1 \leq i \leq n$. Define $f_\lm=\lm h_\lm$. Then $\|f_\lm(\underline{T})\|=\lm \|h_\lm(\underline{T})\| \leq \lm \leq \|f\|_{\infty, W}$ for all $\underline{T} \in \mathfrak{M}_\Hex$ and thus, $\|f_\lm\|_{\mathfrak{M}_\Hex} \leq \lm$. Consequently, $f_\lm \in SA_f(\Hex)$ and $f_\lm(z_i)=\lm_i$ for $1 \leq i \leq n$. By definition of $\rho_f(\textbf{z}, \boldsymbol{\lambda})$, it follows that $\rho \leq \|f_\lm\|_{\mathfrak{M}_\Hex} \leq \lambda$. So, $\rho \leq \lambda$ for every $\lambda \in \Lambda$. Since $\rho \in \lm$, we have that $\rho=\inf \Lambda$. We claim that there exists a matrix $\boldsymbol{\kappa} \in K_\textbf{z}$ and a non-zero vector $y=(y_1, \dotsc, y_n)^t$ such that
		\begin{equation}\label{eqn_302}
			\overset{n}{\underset{i, j=1}{\sum}}(\rho^2-\lm_i\overline{\lm}_j)\boldsymbol{\kappa}(i, j)\overline{y}_iy_j=0.
		\end{equation}
		For $k \in K_{\textbf{z}}$, let $\mu(k)$ be the minimum eigenvalue of $[(\rho^2-\lm_i\overline{\lm}_j)k(i, j)]_{i, j=1}^n$. Set $\mu=\inf_{k \in K_{\textbf{z}}}\mu(k)$. The claim holds trivially for $\mu=0$ since $k \mapsto \mu(k)$ is a continuous map and $K_\textbf{z}$ is compact. Suppose $\mu>0$. Let $\delta=\sup\{\|k\|: k \in K_\textbf{z}\}$, which is finite as $K_\textbf{z}$ is compact. For every $y \in \C^n$ and $k \in K_\textbf{z}$,
		\[
		\left\langle \left[(\rho^2-\lm_i\overline{\lm}_j)k(i, j)\right]_{i, j=1}^ny, y \right\rangle \geq \mu(k)\|y\|^2 \geq \mu \|y\|^2.
		\]
		For $0<\epsilon< \mu \slash \delta$, a routine computation gives that 
		\[
		\left[(\rho^2-\epsilon-\lm_i\overline{\lm}_j)k(i, j)\right]_{i, j=1}^n \geq 0
		\]
		for every $k \in K_\textbf{z}$, contradicting the fact that $\rho=\inf \Lambda$. Hence, $\mu=0$, and the above claim holds for some $\boldsymbol{\kappa} \in K_\textbf{z}$. Since $\boldsymbol{\kappa}=[\boldsymbol{\kappa}(i, j)]_{i, j=1}^n$ is strictly positive, the column space $\mathcal{K}_n=\text{span}\{\boldsymbol{\kappa}(., j) : 1\leq j \leq n \}$ is precisely $n$-dimensional. Consider the operators on $\mathcal{K}_n$ given by
		\[
		S_1^*\boldsymbol{\kappa}(., j)=\overline{z}_j^{(1)}\boldsymbol{\kappa}(., j), \ \ S_2^*\boldsymbol{\kappa}(., j)=\overline{z}_j^{(2)}\boldsymbol{\kappa}(., j), \ \ S_3^*\boldsymbol{\kappa}(., j)=\overline{z}_j^{(3)}\boldsymbol{\kappa}(., j) \ \ \text{and} \ \ S_4^*\boldsymbol{\kappa}(., j)=\overline{z}_j^{(4)}\boldsymbol{\kappa}(., j)
		\]
		for $1 \leq j \leq n$. Also, $(S_1^*, S_2^*, S_3^*, S_4^*)$ is a commuting quadruple of operators with $\sigma_T(S_1^*, S_2^*, S_3^*, S_4^*)=\{(\overline{z}_j^{(1)}, \overline{z}_j^{(2)}, \overline{z}_j^{(3)}, \overline{z}_j^{(4)}) : 1\leq j \leq n\}$, which implies that $\underline{S}=(S_1, S_2, S_3, S_4)$ is subordinate to $\textbf{z}$. In fact, $\underline{S}$ and $\underline{S}^*=(S_1^*, S_2^*, S_3^*, S_4^*)$ belong to $Q\Hex$. Also, $g(\underline{S})^*\boldsymbol{\kappa}(., j)=\overline{g(z_j)}\boldsymbol{\kappa}(., j)=\overline{\lm}_j\boldsymbol{\kappa}(., j)$ for $1 \leq j \leq n$. Since $\left[(\rho^2-\lm_i\overline{\lm}_j)\boldsymbol{\kappa}(i, j)\right] \geq 0$, it follows that $\|g(\underline{S})\|=\|g(\underline{S})^*\| \leq \rho$. The equality $\|g(\underline{S})\|=\rho$ follows directly from \eqref{eqn_302}, which completes the proof.
	\end{proof}
	
	With these preparations, the proof of the extension theorem follows immediately.
	
	\medskip
	
	\noindent \textit{Proof of Theorem \ref{thm_ext_H}:} Let $W \subseteq \Hex$. Suppose $f \in \mathscr{HE}(W)$ is a non-zero function such that 
	\begin{equation}\label{eqn_303}
		\|f(\underline{T})\| \leq \|f\|_{\infty, W} \quad \text{for every $\underline{T} \in Q\Hex$ subordinate to $W$}.	
	\end{equation}
	Choose a dense subset $\{z_1, z_2, \dotsc\}$ of $W$. Let $\textbf{z}_n=\{z_1, \dotsc, z_n\}$ and let $\boldsymbol{\lambda}_n=(f(z_1), \dotsc, f(z_n))$. It follows from Lemma \ref{lem_304} that there exists an extremal function $g_n \in SA_f(\Hex)$ such that $\rho_f(\textbf{z}_n, \boldsymbol{\lambda}_n)=\|g_n\|_{\mathfrak{M}_\Hex}$ for each $n$. By Lemma \ref{lem_305}, there exists $\underline{S}_n=(S_n^{(1)}, S_n^{(2)}, S_n^{(3)}, S_n^{(4)}) \in Q\Hex$ subordinate to $\textbf{z}_n$ such that $\|g_n(\underline{S}_n)\|=\rho_f(\textbf{z}_n, \boldsymbol{\lambda}_n)$. Consequently, 
	\begin{align*}
		\|g_n\|_{\mathfrak{M}_\Hex}
		=\rho_f(\textbf{z}_n, \boldsymbol{\lambda}_n)
		&=\|g_n(\underline{S}_n)\|  \\
		&=\|f(\underline{S}_n)\| \quad [\text{since $\underline{S}_n$ is subordinate to $\textbf{z}$ and $g_n=f$ on $\textbf{z}$}]\\
		&\leq \|f\|_{\infty, W} \quad [\text{$\underline{S}_n$ is subordinate to $W$ since $\textbf{z} \subseteq W$, and $f$ satisfies \eqref{eqn_303}}].
	\end{align*}
	Since $\|g_n\|_{\infty, \Hex} \leq \|g_n\|_{\mathfrak{M}_\Hex} \leq \|f\|_{\infty, W}$, the sequence $\{g_n\}$ is uniformly bounded. By Montel's theorem, one can find a subsequence $\{g_{n_k}\}$ of $\{g_n\}$ that converges pointwise to a function $g \in H^\infty(\Hex)$. Note that  $f(z_i)=g(z_i)$ for $i \in \N$. Therefore, $f=g$ on $W$ and so, $\|f\|_{\infty, W}=\|g\|_{\infty, W} \leq \|g\|_{\infty, \Hex}$. For every $\underline{T}$ in $\mathfrak{M}_\Hex$, an application of dominated convergence theorem gives that $g_{n_k}(\underline{T})$ converges to $g(\underline{T})$ in the weak-operator topology. Thus, $\|g(\underline{T})\| \leq \|f\|_{\infty, W}$ for all $\underline{T} \in \mathfrak{M}_\Hex$ and so, $\|g\|_{\infty, \Hex} \leq \|f\|_{\infty, W}$. The proof is now complete. \qed

	\medskip 
	
	For $W \subseteq \Hex$, we say that $W$ has the \textit{extension property} in $SA(\Hex)$ if for every non-zero $f \in \mathscr{HE}(W)$, there exists $g \in H^\infty(\Hex)$ such that $\frac{1}{\|f\|_{\infty, W}}g \in SA(\Hex), g|_W=f$ and $\|f\|_{\infty, W}=\|g\|_{\infty, \Hex}$. We conclude this section with the following characterization of subsets of $\Hex$ that possess the extension property in $SA(\Hex)$. The proof follows directly from Theorem \ref{thm_ext_H}.
	
	\begin{cor}
		A subset $W$ of $\Hex$ has the extension property in $SA(\Hex)$ if and only if $\|f(\underline{T})\| \leq \|f\|_{\infty, W}$ for every $f \in \mathscr{HE}(W)$ and $\underline{T} \in Q\Hex$ subordinate to $W$. 
	\end{cor}	
	
	\subsection*{Toeplitz corona theorem on the hexablock} 
	One of the key results in this article is the realization theorem for functions in the Schur-Agler class $SA(\Hex)$ for $\Hex$. To do so, the notion of admissible kernels associated with $\Hex$ is introduced. In this direction, we present the vector-valued analog of admissible kernels on $\Hex$. Let $\LS$ be a Hilbert space, and let $F$ be a subset of $\Hex$. A $\mathcal{B}(\LS)$-valued weak kernel of $F$ is a positive semi-definite function $k: F \times F \to \mathcal{B}(\LS)$, that is, 
	\[
	\overset{n}{\underset{i, j=1}{\sum}}\la k(z_i, z_j)v_j, v_i\ra_{\LS} \geq 0 
	\]
	for every finite set $\{z_1, \dotsc, z_n\} \subset F$ and vectors $v_1, \dotsc, v_n \in \LS$. In addition, if $k(z, z) \ne 0$ for all $z \in F$, we say that $k$ is a $\mathcal{B}(\LS)$-valued kernel. A $\mathcal{B}(\LS)$-valued kernel $k$ is said to be \textit{admissible} if
	\begin{enumerate}[leftmargin=*]
	\item $(z, w) \mapsto \left(1-z^{(3)}\overline{w}^{(3)}\right)k(z, w)$ \smallskip 
	\item $(z, w) \mapsto \left(1-\Psi(\al, z^{(2)}, z^{(3)}, z^{(4)})\overline{\Psi(\al, w^{(2)}, w^{(3)}, w^{(4)})}\right)k(z, w)$ \smallskip 
	\item $(z, w) \mapsto \left(1-\psi_{\al_1, \al_2}(z^{(1)}, z^{(2)}, z^{(3)}, z^{(4)})\overline{\psi_{\al_1, \al_2}(w^{(1)}, w^{(2)}, w^{(3)}, w^{(4)}})\right)k(z, w)$
\end{enumerate}
are positive semi-definite maps for all $\alpha \in \DC$ and $(\alpha_1, \alpha_2) \in \DC^2$. The class of scalar-valued admissible kernels on $\Hex$ is denoted by $AK(\Hex)$. For a compact set $K \subseteq \C^n$ and a subset $F$ of $\Hex$, a map $\xi: F \times F \to \mathcal{B}(C(K), \mathcal{B}(\LS))$ is called \textit{completely positive kernel} if for every $n \in \N, \{v_1, \dotsc, v_n\} \subset \LS, \{z_1, \dotsc, z_n\} \subset F$ and $\{h_1, \dotsc, h_n\} \subset C(K)$, we have
	\[
	\overset{n}{\underset{i, j=1}{\sum}}\la \xi(z_i, z_j)(h_i\overline{h}_j)v_j, v_i\ra_\LS \geq 0.
	\]
	If $\LS=\C$, then $\mathcal{B}(C(K), \LS)$ is the dual space $C(K)^*$ of $C(K)$. In this case,  a completely positive kernel $\xi: F \times F \to C(K)^*$ is simply called a \textit{positive kernel}. To prove Toeplitz corona theorem on $\Hex$, we require an operator-valued analogue of the realization theorem on $\Hex$. Since its proof can be simply obtained from the scalar-valued case by the necessary modifications, we shall state the result without proof. We mention that the proof of the scalar-valued realization theorem on $\Hex$. Accordingly, we first establish vector-valued analogues of these two results. As a basic step, we state the following result and its proof follows the same argument as that of Proposition \ref{prop_prelim_I_H}.

	\begin{prop}\label{prop_prelim_I_vv_H}
		Let $F$ be a subset of $\Hex$ and let $K$ be a compact subset of $\C^d$. If $\xi: F \times F \to  \mathcal{B}(C(K), \mathcal{B}(\LS))$ is a completely positive kernel, then there exist a Hilbert space $\HS$, a function $L: F \to \mathcal{B}(C(K), \mathcal{B}(\HS, \LS))$ and a unital $*$-representation $\rho: C(K) \to \mathcal{B}(\HS)$ such that
		\[
		\xi(z, w)(f\overline{g})=L(z)(f)(L(w)(g))^* \quad \text{and} \quad L(z)(fg)^*=\rho(f)^*L(z)(g)^*
		\]	
		for all $f, g \in C(K)$ and $z, w  \in F$. 
	\end{prop}
	
	Let $k_1, k_2$ be $\mathcal{B}(\LS)$-valued functions on $\Hex \times \Hex$. Following the notations in Definition 11.25 of \cite{Agler_McCarthy}, we set $k_1\oslash k_2: \Hex \times \Hex \to \mathcal{B}(\LS \otimes \LS)$ as $(k_1 \oslash k_2)(z, w)=k_1(z, w)\otimes k_2(z, w)$ for $z, w \in \Hex$. We skip its proof as it is a routine adaptation of the scalar-valued arguments presented in Theorem \ref{thm_sa_H}.
	
	\begin{thm}\label{thm_sa_vv_H}
		Let $g: \Hex \times \Hex \to \mathcal{B}(\LS)$ be a self-adjoint function, that is, $g(z, w)=g(w, z)^*$. If $g\oslash k$ is positive semi-definite for any $\mathcal{B}(\LS)$-valued admissible kernel $k$ on $\Hex$, then there exist completely positive kernels $\xi: \Hex \times \Hex \to \mathcal{B}(C(\DC^2), \mathcal{B}(\LS)), \nabla: \Hex \times \Hex \to \mathcal{B}(C(\DC), \mathcal{B}(\LS))$ and a weak $\mathcal{B}(\LS)$-valued kernel $\delta$ on $\Hex$ such that
		\[
		g(z, w)=\xi(z, w)(1-E(z)\overline{E(w)})+\nabla(z, w)(1-e(z)\overline{e(w)})+(1-z^{(3)}\overline{w}^{(3)})\delta(z, w)
		\]
		for all $z, w \in \Hex$. 
	\end{thm}		
	
	Before stating the operator-valued realization theorem for $\Hex$, we introduce a vector-valued analogue of the Schur-Agler class $SA(\Hex)$. For a domain $\Omega$ in $\C^d$ and Hilbert spaces $\LS_1, \LS_2$, let $g: \Omega \to \mathcal{B}(\LS_1, \LS_2)$ be an analytic function. Suppose $\underline{T}=(T_1, \dotsc, T_d)$ is a commuting tuple of operators on a Hilbert space $\LS$ with $\sigma_T(\underline{T}) \subseteq \Omega$. Then $g(\underline{T}): \LS \otimes \LS_1 \to \LS \otimes \LS_2$ is an bounded linear map defined by analytic functional calculus. Let us discuss from \cite{Esch_Put} an equivalent description of $g(\underline{T})$. Given an open neighbourhood $U_0$ of $\sigma_T(\underline{T})$, there is a unique continuous linear map $g \mapsto g(\underline{T})$ from $\text{Hol}(U_0, \mathcal{B}(\LS_1, \LS_2)) \equiv \text{Hol}(U_0)\otimes \mathcal{B}(\LS_1, \LS_2)$ into $\mathcal{B}(\LS \otimes \LS_1, \LS \otimes \LS_2)$ taking $g \otimes B$ into $g(\underline{T}) \otimes B$ for $g \in \text{Hol}(U_0)$ and $B \in \mathcal{B}(\LS_1, \LS_2)$. Therefore, it follows that 
	\[
	\la g(\underline{T})(h_1 \otimes x), h_2 \otimes y\ra=\la g_{x, y}(\underline{T})h_1, h_2 \ra 
	\]
	for every $h_1, h_2 \in \LS, x \in \LS_1$ and $y \in \LS_2$, where $g_{x, y}: \Omega \to \C$ is the scalar-valued holomorphic map defined by $g_{x, y}(z)=\la g(z)x, y\ra$. An interested reader is referred to \cite{Ambrozie, Esch_Put} for further details.

	\smallskip 
	
	For Hilbert spaces $\LS_1, \LS_2$, the \textit{$\mathcal{B}(\LS_1, \LS_2)$-valued Schur-Agler class} $SA_\Hex(\LS_1, \LS_2)$ is the collection of all holomorphic functions $g: \Hex \to \mathcal{B}(\LS_1, \LS_2)$ satisfying $\|g(\underline{T})\| \leq 1$ for every $\underline{T} \in \mathfrak{M}_\Hex$. The following theorem extends the realization theorem established in Theorem \ref{thm_realization_H} to the operator-valued setting. It characterizes the class $SA_{\Hex}(\LS_1,\LS_2)$ in terms of a realization formula. Since the proof is a routine adaptation of that of Theorem \ref{thm_realization_H}, we omit the details.

	\begin{thm}[Vector-valued realization theorem for $\Hex$]\label{thm_realization_vv_H}
		Let $\LS_1$ and  $\LS_2$ be Hilbert spaces. For a function $f: \Hex \to \mathcal{B}(\LS_1, \LS_2)$, the following statements are equivalent: 
		\begin{enumerate}[leftmargin=*]
			\item[$(1)$] $g \in SA_\Hex(\LS_1, \LS_2)$;
			
			\item[$(2)$] $(z, w) \mapsto (I_{\LS_2}-f(z)f(w)^*)\oslash k(z, w)$ is a positive semi-definite function for all $\mathcal{B}(\LS_2)$-valued admissible kernel $k$ on $\Hex$;
			
			\item[$(3)$] there exist a weak $\mathcal{B}(\LS)$-valued kernel $\delta$ on $\Hex$ and completely positive kernels $\xi: \Hex \times \Hex \to \mathcal{B}(C(\DC^2), \mathcal{B}(\LS)), \nabla: \Hex \times \Hex \to \mathcal{B}(C(\DC), \mathcal{B}(\LS))$ such that for all $z, w \in \Hex$,
				\[
				I_{\LS_2}-f(z)f(w)^*=\xi(z, w)(1-E(z)\overline{E(w)})+\nabla(z, w)(1-e(z)\overline{e(w)})+(1-z^{(3)}\overline{w}^{(3)})\delta(z, w),
				\]
				where $E(z)$ and $e(z)$ are as in \eqref{eqn_E(z)};
			
			\item[$(4)$] there exist Hilbert spaces $\HS_1, \HS_2, \HS_3$,  two unital $*$-representations $\rho_1: C(\DC^2) \to \mathcal{B}(\HS_1), \rho_2: C(\DC) \to \mathcal{B}(\HS_2)$ and a unitary $V=\begin{bmatrix}
				A & B \\
				C & D \\
			\end{bmatrix}: \LS_1 \oplus \HS \to \LS_2 \oplus \HS$, where $\HS=\HS_1 \oplus \HS_2 \oplus \HS_3$ such that 
			\[
			f(z)=A+B\begin{bmatrix} \rho_1(E(z)) & 0 & 0\\ 
				0 & \rho_2(e(z)) & 0\\
				0 & 0 & z^{(3)}I_{\HS_3}
			\end{bmatrix} \left(I-D\begin{bmatrix} \rho_1(E(z)) & 0 & 0\\ 
				0 & \rho_2(e(z)) & 0\\
				0 & 0 & z^{(3)}I_{\HS_3}
			\end{bmatrix}\right)^{-1}C.
			\]  
		\end{enumerate}
	\end{thm}

We are now in a position to establish a vector-valued version of the Toeplitz corona theorem for $\Hex$. For given $\varphi_1,\dotsc,\varphi_d\in H^\infty(\Hex)$, Toeplitz corona theorem characterizes the existence of functions $f_1, \dotsc, f_d\in H^\infty(\Hex)$ such that $\varphi_1f_1+\dotsc +\varphi_df_d=1$ with $f_1, \dotsc, f_d$ subjected to a certain norm bound condition. The vector-valued formulation provides a natural extension of this result and the scalar-valued theorem follows as an immediate consequence.

\begin{thm}\label{thm_TC_H}
	Let $\LS_1, \LS_2$ and $\LS_3$ be Hilbert spaces. For functions $f_{12}: \Hex \to \mathcal{B}(\LS_1, \LS_2)$ and $f_{32}: \Hex \to \mathcal{B}(\LS_3, \LS_2)$, the following are equivalent: 
	\begin{enumerate}[leftmargin=*]
		\item[$(1)$] there exists $f_{31} \in SA_\Hex(\LS_3, \LS_1)$ such that $f_{12}(z)f_{31}(z)=f_{32}(z)$ for every $z \in \Hex$;
		\item[$(2)$] for every finite subset $F=\{z_1, \dotsc, z_n\} \subset \Hex$,
		\[
		\begin{bmatrix}
			\left(f_{12}(z_i)f_{12}(z_j)^*-f_{32}(z_i)f_{32}(z_j)^*\right)\otimes k(z_i, z_j)
		\end{bmatrix}_{i, j=1}^n \geq 0
		\]
		for every $\mathcal{B}(\LS_2)$-valued admissible kernel $k$ on $\Hex$;
		\item[$(3)$] there exist a weak $\mathcal{B}(\LS_2)$-valued kernel $\delta$ on $\Hex$ and completely positive kernels $\xi: \Hex \times \Hex \to \mathcal{B}(C(\DC^2), \mathcal{B}(\LS_2)), \nabla: \Hex \times \Hex \to \mathcal{B}(C(\DC), \mathcal{B}(\LS_2))$ such that for all $z, w \in \Hex$,
		\begin{align*}
			&f_{12}(z)f_{12}(w)^*-f_{32}(z)f_{32}(w)^*\\
			&=\xi(z, w)(1-E(z)\overline{E(w)})+\nabla(z, w)(1-e(z)\overline{e(w)})+(1-z^{(3)}\overline{w}^{(3)})\delta(z, w),
		\end{align*}
		where the maps $z\mapsto E(z)$ and $z\mapsto e(z)$ are as in \eqref{eqn_E(z)}.	 	
	\end{enumerate}
\end{thm}	

\begin{proof} The implication $(2)\implies(3)$ follows directly from Theorem \ref{thm_sa_vv_H} applied to the self-adjoint function $(z,w)\mapsto f_{12}(z)f_{12}(w)^*-f_{32}(z)f_{32}(w)^*$. We now prove $(1) \implies (2)$ and $(3) \implies (1)$.

\medskip 

\noindent $(1) \implies (2)$. Let $F=\{z_1, \dotsc, z_n\}$ be a subset of $\Hex$ and let $k$ be a $\mathcal{B}(\LS_2)$-valued admissible kernel on $\Hex$. For $w \in \Hex$ and $h \in \LS_2$, we define $k_wh: \Hex \to \LS_2$ as $k_wh(z)=k(z, w)h$. Let 
	\[
	\HS(k)=\overline{\text{span}}\left\{k_{z}h: z \in \Hex, h \in \LS_2\right\} \ \ \text{and}\  \ \HS_n(k)=\overline{\text{span}}\{k_{z_i}h: h \in \LS_2, 1 \leq i \leq n\}.
	\]
	Clearly, $\HS(k)$ is the reproducing kernel Hilbert space associated with the operator-valued kernel $k$, and is equipped with the inner product $\la k_zh, k_w h'\ra_{\HS(k)}=\la k(w, z)h, h'\ra_{\LS_2}$. Consider the commuting quadruple $\underline{M}=(M_{z^{(1)}}, M_{z^{(2)}}, M_{z^{(3)}}, M_{z^{(4)}})$ consisting of the coordinate multiplication operators on $\HS(k)$. Note that  $M_{z^{(i)}}^*(k_zh)=k_z(\overline{z}^{(i)}h)$ for $ 1 \leq i \leq 4$. For every $\al \in \DC$ and $(\al_1, \al_2) \in \DC^2$.
	\begin{align}\label{eqn_201}
		\|M_{z^{(3)}}\| \leq 1, \ \ \|\Psi(\al, M_{z^{(2)}}, M_{z^{(3)}}, M_{z^{(4)}})\| \leq 1 \ \ \text{and} \ \ \|\psi_{\al_1, \al_2}(M_{z^{(1)}}, M_{z^{(2)}}, M_{z^{(3)}}, M_{z^{(4)}})\| \leq 1
	\end{align}
	 Also, $\HS_n(k)$ is a joint invariant subspace for $\underline{M}^*=(M_{z^{(1)}}^*, M_{z^{(2)}}^*, M_{z^{(3)}}^*,  M_{z^{(4)}}^*)$ since $M_{z^{(i)}}^*(k_{z_j}h)=k_{z_j}(\overline{z}_j^{(i)}h)$ for $1 \leq i \leq 4$ and $1 \leq j \leq n$. Consider  $\underline{T}=(T_1, T_2, T_3, T_4)$ on $\HS_n(k)$ given by
	\[
	(T_1^*, T_2^*, T_3^*, T_4^*)=\left(M_{z^{(1)}}^*|_{\HS_n(k)}, M_{z^{(2)}}^*|_{\HS_n(k)}, M_{z^{(3)}}^*|_{\HS_n(k)}, M_{z^{(4)}}^*|_{\HS_n(k)}\right).
	\]
	Since $\Hex$ is a $(1, 1, 1, 2)$-quasi-balanced domain, it follows from \eqref{eqn_201} that the quadruple $s\cdot\underline{T}^*=(sT_1^*, sT_2^*, sT_3^*, s^2T_4^*) \in \mathfrak{M}_\Hex$ for every $s \in (0, 1)$. Take a holomorphic map $f: \overline{\Hex} \to \mathcal{B}(\LS_3, \LS_1)$ such that $f \in SA_\Hex(\LS_3, \LS_1)$ and define $\widehat{f}: \overline{\Hex} \to \mathcal{B}(\LS_1, \LS_3)$ as $\widehat{f}(z)=f(\overline{z})^*$ for $z \in \overline{\Hex}$. Then $\widehat{f} \in SA_\Hex(\LS_1, \LS_3)$. Evidently, $\widehat{f}(s\cdot \underline{T}^*) \in \mathcal{B}(\HS_n(k)\otimes \LS_1, \HS_n(k)\otimes \LS_3)$ that satisfies
	\begin{align}\label{eqn_202}
		\widehat{f}(s\cdot \underline{T}^*)(k_{z_j}u\otimes v)=k_{z_j}u\otimes f(s\cdot z_j)^*v
	\end{align} 
	for every $u \in \LS_2, v \in \LS_1$ and $1 \leq j \leq n$. To see this, note that
	\begin{align*}
		\la \widehat{f}(s\cdot \underline{T}^*)(k_{z_j}u_j\otimes v_j), k_{z_i}u_i\otimes v_i \ra 	
		&=\la \widehat{f}_{v_j, v_i}(s\cdot \underline{T}^*)k_{z_j}u_j, k_{z_i}u_i \ra \\
		&=\la k_{z_j}(\widehat{f}_{v_j, v_i}(s\cdot \overline{z}_j)u_j), k_{z_i}u_i\ra \quad [M_{z^{(i)}}^*(k_{z_j}h)=k_{z_j}(\overline{z}_j^{(i)}h) \ \text{for} \ 1 \leq i \leq 4]\\
		&=  \la k(z_i, z_j)(\widehat{f}_{v_j, v_i}(s\cdot \overline{z}_j)u_j), u_i\ra \\
		&=\widehat{f}_{v_j, v_i}(s\cdot \overline{z}_j) \la k(z_i, z_j)u_j, u_i\ra \\
		&=\la \widehat{f}(s\cdot \overline{z}_j)v_j, v_i \ra \la k(z_i, z_j)u_j, u_i\ra\\
		&=\la f(s\cdot z_j)^*v_j, v_i \ra \la k(z_i, z_j)u_j, u_i\ra\\
		&= \la k_{z_j}u_j \otimes f(s\cdot \overline{z}_j)^*v_j, k_{z_i}u_i \otimes v_i \ra
	\end{align*}
	for every $u_i, u_j \in \LS_2, v_i \in \LS_3$ and $v_j \in \LS_1$ with $ 1\leq i, j \leq n$.
	 Let $\{u_i, v_i : 1 \leq i \leq n\} \subset \LS_2$. Since $\widehat{f} \in SA_\Hex(\LS_1, \LS_3)$ and $s\cdot \underline{T}^* \in \mathfrak{M}_\Hex$, it follows that $\|\widehat{f}(s\cdot \underline{T}^*)\| \leq 1$. Then
	\begin{align*}
		0 & \leq \left\la \left(I_{\HS_n(k)\otimes \HS_1}-\widehat{f}(s\cdot \underline{T}^*)^*\widehat{f}(s\cdot \underline{T}^*)\right) \left[\overset{n}{\underset{j=1}{\sum}}k_{z_j}u_j\otimes f_{12}(z_j)^*v_j  \right], \left[\overset{n}{\underset{i=1}{\sum}}k_{z_i}u_i\otimes f_{12}(z_i)^*v_i  \right]\right\ra \\
		&= \overset{n}{\underset{i, j=1}{\sum}}\left \la k_{z_j}u_j\otimes f_{12}(z_j)^*v_j, k_{z_i}u_i\otimes f_{12}(z_i)^*v_i \right \ra \\
		& \quad -
		\overset{n}{\underset{i, j=1}{\sum}}\left \la \widehat{f}(s\cdot \underline{T}^*)\left[k_{z_j}u_j\otimes f_{12}(z_j)^*v_j\right], \widehat{f}(s\cdot \underline{T}^*)\left[k_{z_i}u_i\otimes f_{12}(z_i)^*v_i\right] \right \ra \\
		&= \overset{n}{\underset{i, j=1}{\sum}} \left \la k(z_i, z_j)u_j, u_i \right \ra
		\left \la \left(f_{12}(z_i)f_{12}(z_j)^*-f_{12}(z_i)f(s\cdot z_i) f(s\cdot z_j)^*f_{12}(z_j)^*\right)v_j,v_i \right \ra
	\end{align*}
	and taking the limit $s \to 1^-$, it follows that
	\begin{align}\label{eqn_203}
		\overset{n}{\underset{i, j=1}{\sum}} \left \la k(z_i, z_j)u_j, u_i \right \ra
		\left \la \left(f_{12}(z_i)f_{12}(z_j)^*-f_{12}(z_i)f(z_i) f(z_j)^*f_{12}(z_j)^*\right)v_j,v_i \right \ra \geq 0.
	\end{align}
	For $f_{31}\in SA_{\Hex}(\LS_3,\LS_1)$, apply the above argument to the functions
	$f_{31}^{(r)}(z)=f_{31}(r\cdot z)$, where  $0<r<1$. Letting $r\to1^-$ gives that
		\begin{align*}
		\overset{n}{\underset{i, j=1}{\sum}} \left \la k(z_i, z_j)u_j, u_i \right \ra
		\left \la \left(f_{12}(z_i)f_{12}(z_j)^*-f_{12}(z_i)f_{31}(z_i) f_{31}(z_j)^*f_{12}(z_j)^*\right)v_j,v_i \right \ra \geq 0.
	\end{align*}

	\medskip

	\noindent	$(3) \implies (1)$. Assume that there exist a weak $\mathcal{B}(\LS_2)$-valued kernel $\delta$ on $\Hex$ and completely positive kernels $\xi: \Hex \times \Hex \to \mathcal{B}(C(\DC^2), \mathcal{B}(\LS_2)), \nabla: \Hex \times \Hex \to \mathcal{B}(C(\DC), \mathcal{B}(\LS_2))$ such that for all $z, w \in \Hex$,
	\begin{align}\label{eqn_2004}
		&f_{12}(z)f_{12}(w)^*-f_{32}(z)f_{32}(w)^* \notag \\
		&=\xi(z, w)(1-E(z)\overline{E(w)})+\nabla(z, w)(1-e(z)\overline{e(w)})+(1-z^{(3)}\overline{w}^{(3)})\delta(z, w).
	\end{align}
	By Proposition \ref{prop_prelim_I_vv_H}, there exist Hilbert space $\KS_1, \KS_2$ and maps $L_1: \Hex \to \mathcal{B}(C(\DC^2), \mathcal{B}(\KS_1, \LS_2))$ and $L_2: \Hex \to \mathcal{B}(C(\DC), \mathcal{B}(\KS_2, \LS_2))$ such that for every $z, w \in \Hex$, $f_1, h_1 \in C(\DC^2)$ and $f_2, h_2 \in C(\DC)$,
	\begin{align}\label{eqn_2005}
		\xi(z, w)(f_1\overline{h}_1)=L_1(z)(f_1)(L_1(w)(h_1))^* \quad  \text{and} \quad \nabla(z, w)(f_2\overline{h}_2)=L_2(z)(f_2)(L_2(w)(h_2))^*.
	\end{align} 
	We also have unital $*$-representations $\rho_1: C(\DC^2)\to \mathcal{B}(\KS_1)$ and $\rho_2: C(\DC)\to \mathcal{B}(\KS_2)$ that satisfy 
	\begin{align}
		(L_1(z)(f_1h_1))^*=\rho_1(f_1)^*(L_1(z)(h_1))^* \quad \text{and} \quad (L_2(z)(f_2h_2))^*=\rho_2(f_2)^*(L_2(z)(h_2))^*
	\end{align}
for every $z, w \in \Hex$, $f_1, h_1 \in C(\DC^2)$ and $f_2, h_2 \in C(\DC)$. By Theorem 2.62 in \cite{Agler_McCarthy}, there exist a Hilbert space $\KS_3$ and a map $g: \Hex \to \mathcal{B}(\KS_3, \LS_2)$ such that $\delta(z, w)=g(z)g(w)^*$. Set $\KS=\KS_1 \oplus \KS_2 \oplus \KS_3$. Using \eqref{eqn_2005}, we can re-write \eqref{eqn_2004} as 
	\begin{align*}
		&f_{12}(z)f_{12}(w)^*+L_1(z)(E(z))(L_1(w)(E(w)))^*+L_2(z)(e(z))(L_2(w)(e(w)))^*+z^{(3)}\overline{w}^{(3)}g(z)g(w)^*
		\\
		&=
		f_{32}(z)f_{32}(w)^*+L_1(z)(1)(L_1(w)(1))^*+L_2(z)(1)(L_2(w)(1))^*+g(z)g(w)^*
	\end{align*}
	for every $z, w \in \Hex$. Consider the spaces given by
	{\small 
		\setlength{\arraycolsep}{3pt}
		\begin{align*}
			\LS_d
			= \overline{\text{span}} \left\{
			\begin{bmatrix}
				f_{12}(z)^*v \\
				Y(z)^*
				\begin{bmatrix}
					(L_1(z)(1))^*v \\
					(L_2(z)(1))^*v \\
					g(z)^*v
				\end{bmatrix}
			\end{bmatrix}
			: z \in \Hex, v \in \LS_2
			\right\} \ \text{and} \ \LS_r
			= \overline{\text{span}} \left\{
			\begin{bmatrix}
				f_{32}(z)^*v \\
				(L_1(z)(1))^*v \\
				(L_2(z)(1))^*v \\
				g(z)^*v
			\end{bmatrix}
			: z \in \Hex, v \in \LS_2
			\right\},
		\end{align*}
	}
where  $Y(z)=\begin{bmatrix} \rho_1(E(z)) & 0 & 0\\ 
	0 & \rho_2(e(z)) & 0\\
	0 & 0 & z^{(3)}I_{\LS_3}
\end{bmatrix}$.	Evidently, $\LS_d$ and $\LS_r$ are subspaces of $\LS_1 \oplus \KS$ and $\LS_3 \oplus \KS$, respectively. Furthermore, the linear operator $V: \LS_d \to \LS_r$ defined by
	\[
	V\begin{bmatrix}
		f_{12}(z)^*v \\
		Y(z)^*
		\begin{bmatrix}
			(L_1(z)(1))^*v \\
			(L_2(z)(1))^*v \\
			g(z)^*v
		\end{bmatrix}
	\end{bmatrix}
	=\begin{bmatrix}
		f_{32}(z)^*v \\
		(L_1(z)(1))^*v \\
		(L_2(z)(1))^*v \\
		g(z)^*v
	\end{bmatrix}
	\]
	is an isometry. Consequently, $V$ extends to a unitary from $\LS_1 \oplus \KS$ onto $\LS_3 \oplus \KS$. If we write $V=\begin{bmatrix} A_1 & B_1 \\ C_1 & D_1 \end{bmatrix}$, then $f_{32}(z)^*v=A_1f_{12}(z)^*v+B_1Y(z)^*(I-D_1Y(z)^*)^{-1}C_1f_{12}(z)^*v$ for $z \in \Hex$ and $v \in \LS_2$. So, we define  $f_{31}: \Hex \to \mathcal{B}(\LS_3, \LS_1)$ as 
	$f_{31}(z)^*=A_1+B_1Y(z)^*(I-D_1Y(z)^*)^{-1}C_1$. Clearly, $f_{12}(z)f_{31}(z)=f_{32}(z)$ for every $z \in \Hex$ and  
	\[
	f_{31}(z)=A_1^*+C_1^*Y(z)(I-D_1^*Y(z))^{-1}B_1^*=A+BY(z)(I-DY(z))^{-1}C.
	\]
It follows from Theorem \ref{thm_realization_vv_H} that $f_{31} \in SA_{\Hex}(\LS_3, \LS_1)$. The proof is now complete. 
\end{proof}

Finally, we present the Toeplitz corona theorem on $\Hex$ and conclude the section. 

\begin{thm}\label{thm_TC_H_II}
	For $\varphi_1, \dotsc, \varphi_d \in H^\infty(\Hex)$ and $\epsilon>0$, the following statements are equivalent:
	\begin{enumerate}[leftmargin=*]
		\item[$(1)$] there exist $f_1, \dotsc, f_d \in H^\infty(\Hex)$ such that  
		$
		\begin{bmatrix}
			f_1 & \dotsc & f_d
		\end{bmatrix}^t \in \frac{1}{\epsilon}SA_\Hex(\C, \C^d)
		$ 
		and $\overset{d}{\underset{j=1}{\sum}}\varphi_jf_j=1$;
		
		\item[$(2)$]
		$
		(z, w) \mapsto \left(\overset{d}{\underset{j=1}{\sum}} \varphi_j(z)\overline{\varphi_j(w)}-\epsilon^2\right)k(z, w)
		$ on $\Hex \times \Hex$	is positive semi-definite for every $k \in AK(\Hex)$;
		
		\item[$(3)$] there exist $\xi \in C(\overline{\D}^2)^+_\Hex, \nabla \in C(\DC)^+_\Hex$ and $\delta \in \C_\Hex^+$ such that for all $z, w \in \Hex$,
		\[
		\overset{d}{\underset{j=1}{\sum}} \varphi_j(z)\overline{\varphi_j(w)}-\epsilon^2=\xi(z, w)(1-E(z)\overline{E(w)})+\nabla(z, w)(1-e(z)\overline{e(w)})+(1-z^{(3)}\overline{w}^{(3)})\delta(z, w),
		\]
		where $E(z)$ and $e(z)$ are as in \eqref{eqn_E(z)}.
	\end{enumerate}
\end{thm}	

\begin{proof}
	The proof follows from Theorem \ref{thm_TC_H} by choosing the spaces $\LS_1=\C^d, \LS_2=\LS_3=\C$, and defining $f_{32}(z)=\epsilon, f_{12}(z)=\begin{bmatrix}
		\varphi_1(z) & \dotsc & \varphi_d(z)
	\end{bmatrix}$ and $f_{31}(z)=\epsilon \begin{bmatrix}
		f_1(z) & \dotsc & f_d(z)
	\end{bmatrix}^t$ for $z \in \Hex$.
\end{proof}

	\section{Realization, interpolation, extension and Toeplitz corona theorems on $\B_2$}\label{sec_biball}

	\noindent For $z=(z_1, \dotsc, z_d)$ and $w=(w_1, \dotsc, w_d)$ in $\C^d$, set $\langle z, w \rangle=z_1\overline{w}_1+\dotsc+z_d\overline{w}_d$ and  
	$
	k_d(z, w)=1\slash (1-\langle z, w\rangle).
	$ 
	The Drury-Arveson space is the Hilbert space of analytic functions on $\B_d$ with reproducing kernel $k_d(z, w)$, e.g., see \cite{Arveson, Drury}. An analytic function $\varphi(z)$ on $\B_d$ is a multiplier of the Drury-Arveson space if the multiplication by $\varphi(z)$ is a bounded operator on the Drury-Arveson space. A multiplier is said to be contractive if the corresponding multiplication operator is a contraction. The class of contractive multipliers of the Drury-Arveson space is called the Schur-Agler class on the unit ball, and is a proper subset of the Schur class $S(\B_d)$ when $d>1$. An interested reader is referred to \cite{Agler_McCarthy, Ball_Fang, Esch_Put} for further details on the Schur-Agler class on $\B_d$ and the associated realization theorem. A version of Nevanlinna-Pick type interpolation in the setting of Drury-Arveson space to a large class of vector-valued reproducing kernel Hilbert spaces over $\B_d$ is studied in \cite{Deepak_Jaydeb_II} (also see \cite{Deepak_Jaydeb, Kosinski_II}, and the references therein). In this section, we study a different subclass of Schur class $S(\B_2)$, and provide suitable versions of realization, interpolation and extension theorems on $\B_2$. In this direction, we recall the following result from \cite{Biswas}.
	
	\begin{thm}[\cite{Biswas}, Theorem 10.14]\label{thm_connect_H_B2}
		A point $(z^{(1)}, z^{(2)}) \in \mathbb{B}_2$ if and only if $(z^{(1)}, z^{(2)}, 0, 0) \in \Hex$.
	\end{thm} 
	
	The above theorem shows that $\mathbb{B}_2$ can be embedded into $\mathbb{H}$, which allows us to establish the realization, interpolation and extension results on $\mathbb{B}_2$ through the corresponding results on $\Hex$. This shows that the function theory on $\Hex$ naturally provides a single framework from which the aforementioned theorems on $\B_2$ can be derived. Motivated by Theorem \ref{thm_connect_H_B2}, we consider  
	\begin{align*}
		\mathfrak{M}_{\B_2}&=\{(T_1, T_2): (T_1, T_2, 0, 0) \in \mathfrak{M}_\Hex\}.
	\end{align*}
	Evidently, $(z^{1)}I, z^{(2)}I) \in \mathfrak{M}_{\B_2}$ for all $(z^{1)}, z^{(2)}) \in \B_2$. Consequently, the class given by 
	\[
	\mathfrak{C}(\B_2)=\{g \in \text{Hol}(\B_2): \|g(T_1, T_2)\|\leq 1 \ \text{for all} \ (T_1, T_2) \in \mathfrak{M}_{\B_2} \}
	\]
	is contained in the Schur class $S(\B_2)$. 	It is easy to see that a holomorphic map $g: \mathbb B_2 \to \C$ induces a holomorphic map $g\circ \theta_{\Hex \to \mathbb B_2}$ on $\Hex$, where $\theta_{\Hex \to \B_2}: \Hex \to \B_2, (z^{(1)}, z^{(2)}, z^{(3)}, z^{(4)}) \mapsto (z^{(1)}, z^{(2)})$. Capitalizing this, we present the following realization theorem for functions in $ \mathfrak{C}(\B_2)$.
	
	\begin{thm}\label{thm_realization_H_B2}
		For a function $g : \B_2 \to \C$, the following are equivalent: 
		\begin{enumerate}[leftmargin=*]
			\item[$(1)$] $g \in \mathfrak{C}(\B_2)$;
			\item[$(2)$] $f=g\circ \theta_{\Hex \to \B_2} \in SA(\Hex)$;
			\item[$(3)$] $(1-f(z)\overline{f(w)})k(z, w) \succcurlyeq 0$ for all $k \in AK(\Hex)$;
			\item[$(4)$] there exist $\xi \in C(\overline{\D}^2)^+_\Hex, \nabla \in C(\DC)^+_\Hex$ and $\delta \in \C_\Hex^+$ such that for all $z, w \in \Hex$,
			\[
			1-f(z)\overline{f(w)}=\xi(z, w)(1-E(z)\overline{E(w)})+\nabla(z, w)(1-e(z)\overline{e(w)})+(1-z^{(3)}\overline{w}^{(3)})\delta(z, w),
			\]
			where $E(z)$ and $e(z)$ are as in \eqref{eqn_E(z)};
			\item[$(5)$] $f \in UC(\Hex)$.
		\end{enumerate}
	\end{thm}
	
	\begin{proof} The part $(2) \implies (1)$ follows from the definition of $\mathfrak{C}(\B_2)$, and the implications $(2) \implies (3) \implies (4) \implies (5) \implies (2)$ are consequences of Theorem \ref{thm_realization_H}. It remains to show that $(1) \implies (2)$. Let $g \in \mathfrak{C}(\B_2)$ and let $\underline{T}=(T_1, T_2, T_3, T_4) \in \mathfrak{M}_\Hex$. Then $\|T_3\|<1, \|\Psi(\al, T_2, T_3, T_4)\|<1$ and $\|\psi_{\al_1, \al_2}(T_1, T_2, T_3, T_4)\|<1$ for all $\al \in \DC$ and $(\al_1, \al_2) \in \DC^2$. Let $\al, \al_1, \al_2 \in \DC$. Then $\|\Psi(\alpha, T_2, 0, 0)\|=\|T_2\|=\|\Psi(0, T_2, T_3, T_4)\|<1$ and
		\begin{align*}
			\|\psi_{\al_1, \al_2}(T_1, T_2, 0, 0)\|
			=\sqrt{(1-|\al_1|^2)(1-|\al_2|^2)}\|T_1(I-\al_1T_2)^{-1}\| 
			& \leq \sqrt{1-|\al_1|^2} \ \|T_1(I-\al_1T_2)^{-1}\| \\
			&=\|\psi_{\al_1, 0}(T_1, T_2, T_3, T_4)\|\\
			&<1.
		\end{align*}
		Therefore, $(T_1, T_2, 0, 0) \in \mathfrak{M}_\Hex$ and so, $(T_1, T_2) \in \mathfrak{M}_{\B_2}$. Then $\|g\circ \theta_{\Hex \to \B_2} (\underline{T})\|=\|g(T_1, T_2)\| \leq 1$, which shows that $g \circ \theta_{\Hex \to \B_2} \in SA(\Hex)$. 
	\end{proof}
	
	An application of Theorem \ref{thm_realization_H} provides the following interpolation theorem on $\B_2$.
	
	\begin{thm}\label{thm_interpolation_H_B2}
		Let $F=\{z_1, \dotsc, z_n\} \subseteq \B_2$ and let $\lm_1, \dotsc, \lm_n \in \DC$. Then the following are equivalent:
		\begin{enumerate}[leftmargin=*]
			\item[$(1)$] there exists $g \in \mathfrak{C}(\B_2)$ such that $g(z_j)=\lm_j$ for $1 \leq j \leq n$;
			\item[$(2)$] there exists $f \in SA(\Hex)$ such that $f(z_j^{(1)}, z_j^{(2)}, 0, 0)=\lm_j$, where $z_j=(z_j^{(1)}, z_j^{(2)})$ for $1 \leq j \leq n$;
			\item[$(3)$] $\begin{bmatrix} (1-\lm_i\overline{\lm_j})k\left((z_i^{(1)}, z_i^{(2)}, 0, 0), (z_j^{(1)}, z_j^{(2)}, 0, 0)\right) \end{bmatrix}_{i, j=1}^n \geq 0$ for all $k \in AK(\Hex)$;
			\item[$(4)$] there exist $\xi \in C(\overline{\D}^2)^+_F$ and $\nabla \in C(\DC)^+_F$ such that for $w_i=(z_j^{(1)}, z_j^{(2)}, 0, 0)$ and $1\leq i, j\leq  n$,
			\[
			1-\lambda_i\overline{\lambda}_j=\xi(w_i, w_j)(1-E(w_i)\overline{E(w_j)})+\nabla(w_i, w_j)(1-e(w_i)\overline{e(w_j)}).
			\] 	
		\end{enumerate}
		
		In the case when $(1)$ holds, one may choose $f=g\circ \theta_{\Hex \to \B_2}$ in $(2)$.
	\end{thm}
	
	\begin{proof}
		The implication $(1) \implies (2)$ follows directly from Theorem \ref{thm_realization_H_B2} by choosing $f=g\circ \theta_{\Hex \to \B_2}$. The equivalence of $(2)$ with $(3)-(4)$ follows from Theorem \ref{thm_interpolation_H}. It is only left to prove $(2) \implies (1)$. Suppose there exists $f \in SA(\Hex)$ such that $f(z_j^{(1)}, z_j^{(2)}, 0, 0)=\lm_j$ for $1 \leq j \leq n$. Define $g: \B_2 \to \C$ as $g(z^{(1)}, z^{(2)})=f(z^{(1)}, z^{(2)}, 0, 0)$. For $(T_1, T_2) \in \mathfrak{M}_{\B_2}$, we have that  $(T_1, T_2, 0, 0) \in \mathfrak{M}_\Hex$ and so, $\|g(T_1, T_2)\|=\|f(T_1, T_2, 0, 0)\| \leq 1$. Thus, $g \in \mathfrak{C}(\B_2)$ and $g(z_j)=\lambda_j$ for $ 1\leq j \leq n$.  
	\end{proof}
	
	Recall that a subset $W$ of $\Hex$ is said to have the extension property in $SA(\Hex)$ if for every non-zero $f \in \mathscr{HE}(W)$, there exists a norm-preserving extension $g \in H^\infty(\Hex)$ of $f$ such that $g \slash \|f\|_{\infty, W} \in SA(\Hex)$. The biball $\mathbb{B}_2$ can be holomorphically embedded in $\Hex$ as $\B_2 \times \{0\} \times \{0\}$, which has the extension property in $SA(\Hex)$ as our next example shows.

	\begin{eg}\label{eg_002}
		Let $W=\B_2 \times \{0\} \times \{0\}$. By Theorem \ref{thm_connect_H_B2}, $W \subseteq \Hex$.  Let $f \in \mathscr{HE}(W)$ be non-zero. It follows from Lemma 6.5 in \cite{Biswas} that $(w^{(1)}, w^{(2)}) \in \B_2$ for every $(w^{(1)}, w^{(2)}, w^{(3)}, w^{(4)}) \in \Hex$. Consider the maps $g: \Hex \to \C$ and $h: \B_2 \to \C$ given by $g(w^{(1)}, w^{(2)}, w^{(3)}, w^{(4)})=f(w^{(1)}, w^{(2)}, 0, 0)$ and $h(w^{(1)}, w^{(2)})=f(w^{(1)}, w^{(2)}, 0, 0)$. It is clear that $g$ and $h$ are holomorphic maps on $\Hex$ and $\B_2$, respectively. Also, $g|_W=f$ and $\|g\|_{\infty, \Hex}=\|f\|_{\infty, W}$. Let $\underline{T}=(T_1, T_2, T_3, T_4) \in \mathfrak{M}_\Hex$ be acting on a Hilbert space $\mathcal{K}$. Then $\|T_3\|<1, \|\Psi(\al, T_2, T_3, T_4)\|<1$ and $\|\psi_{\al_1, \al_2}(T_1, T_2, T_3, T_4)\|<1$ for $\al \in \DC$ and $(\al_1, \al_2) \in \DC^2$. Let $\al, \al_1, \al_2 \in \DC$. Then $\|\Psi(\alpha, T_2, 0, 0)\|=\|T_2\|=\|\Psi(0, T_2, T_3, T_4)\|<1$ and
		\begin{align*}
			\|\psi_{\al_1, \al_2}(T_1, T_2, 0, 0)\|
			\leq \sqrt{1-|\al_1|^2} \ \|T_1(I-\al_1T_2)^{-1}\| 
			=\|\psi_{\al_1, 0}(T_1, T_2, T_3, T_4)\|
			<1.
		\end{align*}
		Thus, $(T_1, T_2, 0, 0) \in \mathfrak{M}_\Hex$ and so, $(T_1, T_2) \in \mathfrak{M}_{\B_2}$. For $r \in (0, 1)$, set $h_r(w^{(1)}, w^{(2)})=h(rw^{(1)}, rw^{(2)})$ for $(w^{(1)}, w^{(2)})\in \overline{\B}_2$. Clearly, $h_r \in \text{Hol}(\overline{\B}_2)$. Then
		\begin{align*}
			\|g(rT_1, rT_2, rT_3, r^2T_4)\|=\|h(rT_1, rT_2)\| =\|h_r(T_1, T_2)\|\leq \|h_r\|_{\infty, \overline{\B}_2} \leq \|h\|_{\infty, \B_2} \leq \|f\|_{\infty, W}
		\end{align*}
		for every $r \in (0, 1)$. Letting $r\to 1$, we have that $\|g(\underline{T})\| \leq \|f\|_{\infty, W}$ and thus, $\frac{1}{\|f\|_{\infty, W}}g \in SA(\Hex)$. Consequently, $W$ has the extension property in $SA(\Hex)$. \qed 
	\end{eg}
	
	Let $V$ be a subset of $\B_2$. A commuting pair $(A, B)$ of operators is said to be \textit{subordinate} to $V$, if $\sigma_T(A, B) \subset V$, and $g(A, B)=0$ whenever $g$ is holomorphic in a neighbourhood of $V$ and $g|_V=0$. If $f$ is a function on $V$ that admits a holomorphic extension in a neighbourhood of $V$ and $(A, B)$ is subordinate to $V$, then define $f(A, B)=g(A, B)$, where $g$ is any holomorphic extension of $f$ in a neighbourhood of $W$. Also, set $Q\B_2=\{(T_1, T_2) : (T_1, T_2, 0, 0) \in Q\Hex\}$.
	
	\begin{thm}\label{thm_ext_H_B2}
		Let $V \subseteq \B_2$ and let $h \in \mathscr{HE}(V)$. Then there is a bounded holomorphic function $\widetilde{h}$ on $\B_2$ such that $\widetilde{h}|_V=h$ and $\|h\|_{\infty, V}=\|\widetilde{h}\|_{\infty, \B_2}$ if and only if $\|h(A, B)\| \leq \|h\|_{\infty, V}$ for every $(A, B) \in Q\B_2$ subordinate to $V$.
	\end{thm}

	\begin{proof} 
		Set $W=\{(z^{(1)}, z^{(2)}, 0, 0): (z^{(1)}, z^{(2)}) \in V\}$ and $W_0=\{(z^{(1)}, z^{(2)}, 0, 0): (z^{(1)}, z^{(2)}) \in \B_2\}$. We have by Theorem \ref{thm_connect_H_B2} that $W \subseteq W_0 \subseteq \Hex$. Define $f: W \to \C$ as $f(z^{(1)}, z^{(2)}, 0, 0)=h(z^{(1)}, z^{(2)})$. Since $h \in \mathscr{HE}(V)$ extends to a holomorphic map $h_0$ on a neighbourhood $U$ of $V$, $F(z^{(1)},z^{(2)},z^{(3)}, z^{(4)}) = h_0(z^{(1)},z^{(2)})$ is a holomorphic extension of $f$ to $U \times \C^2$ of $W$ such that $F|_W=f$ and so, $f \in \mathscr{HE}(W)$. 
		
		\smallskip 
		
		We start with the proof of the forward implication. Suppose there exists a bounded holomorphic function $\widetilde{h}$ on $\B_2$ such that $\widetilde{h}|_V=h$ and $\|h\|_{\infty, V}=\|\widetilde{h}\|_{\infty, \B_2}$. The map $\widetilde{f}: W_0 \to \C$ given by $\widetilde{f}(z^{(1)}, z^{(2)}, z^{(3)},z^{(4)})=\widetilde{h}(z^{(1)}, z^{(2)})$ has a holomorphic extension on $\Hex$ as $\widetilde{F}(z^{(1)}, z^{(2)}, z^{(3)}, z^{(4)})= \widetilde{h}(z^{(1)}, z^{(2)})$ is holomorphic on $\Hex$ that satisfies $\widetilde{F}|_{W_0}=\widetilde{f}$. Thus, $\widetilde{f} \in \mathscr{HE}(W_0)$. It follows from Example \ref{eg_002} that $W_0$ has the extension property in $SA(\Hex)$. Then 
		\begin{equation}\label{eqn_4003}
			\|\widetilde{f}(\underline{T})\| \leq \|\widetilde{f}\|_{\infty, W_0}\leq \|\widetilde{h}\|_{\infty, \B_2}=\|h\|_{\infty, V}  
		\end{equation}
		for every $\underline{T} \in Q\Hex$ subordinate to $W_0$. Suppose $(T_1, T_2) \in Q\B_2$ is a commuting pair of operators subordinate to $V$. Then $(T_1, T_2, 0, 0) \in Q\Hex$ and $\sigma_T(T_1, T_2, 0, 0)= \sigma_T(T_1, T_2) \times \{0\} \times \{0\}\subseteq V\times \{0\} \times \{0\} = W$. Let $G$ be holomorphic on an open neighbourhood $U \subseteq \mathbb{C}^4$ of $W$ such that $G|_W = 0$. Define $\xi : \mathbb{C}^2 \to \mathbb{C}^4$ as $\xi(z^{(1)}, z^{(2)}) = (z^{(1)}, z^{(2)}, 0, 0)$ and set $g(z^{(1)}, z^{(2)})= G(z^{(1)}, z^{(2)}, 0, 0)$ for $(z^{(1)}, z^{(2)}) \in \xi^{-1}(U)$. Clearly, $g$ is holomorphic in a neighbourhood of $V$. For $(z^{(1)}, z^{(2)}) \in V$, we have that $g(z^{(1)}, z^{(2)}) = G(z^{(1)}, z^{(2)}, 0, 0) = 0$ as $G|_W=0$. Since $(T_1, T_2)$ is subordinate to $V$ and $g|_V=0$, it follows that
		$g(T_1, T_2) = 0$. By holomorphic functional calculus, $G(T_1, T_2, 0, 0) = g(T_1, T_2)= 0$. Hence, $(T_1, T_2, 0, 0)$ is subordinate to $W$ and so, $(T_1, T_2, 0, 0)$ is subordinate to $W_0$ since $W \subseteq W_0$. By \eqref{eqn_4003}, 	for every commuting pair $(T_1, T_2)$ subordinate to $V$, we have
		\[
		\|h(T_1, T_2)\|=\|\widetilde{h}(T_1, T_2)\|=\|\widetilde{F}(T_1, T_2, 0, 0)\|=\|\widetilde{f}(T_1, T_2, 0, 0)\| \leq \|h\|_{\infty, V}
		\]
		Conversely, let $\|h(T_1, T_2)\| \le \|h\|_{\infty,V}$ for every$(T_1, T_2) \in Q\B_2$ subordinate to $V$. Suppose $\underline{T} = (T_1, T_2, T_3, T_4) \in Q\Hex$ is subordinate to $W$. Clearly, $(rT_1, rT_2, rT_3, r^2T_4) \in \mathfrak{M}_\Hex$ for $0<r<1$. Following similar computations as in Example \ref{eg_002}, we have that $(rT_1, rT_2) \in \mathfrak{M}_{\B_2}$ for $0<r<1$. Letting $r \to 1$, it follows that $(T_1, T_2) \in Q\B_2$. Also, $\sigma_T(\underline{T}) \subseteq W$.
		Let $G_j(z^{(1)},z^{(2)},z^{(3)}, z^{(4)}) = z^{(j)}$ for $j=3,4$. Clearly, $G_3, G_4$ are holomorphic on $\C^4$ and $G_3|_W=G_4|_W = 0$. Since $\underline{T}$ is subordinate to $W$, it follows that $T_j= G_j(T_1,T_2,T_3, T_4) = 0$ for $j=3,4$ and so, $\underline{T} = (T_1,T_2, 0, 0)$. We show that $(T_1, T_2)$ is subordinate to $V$. By spectral mapping principle, $\sigma_T(T_1, T_2) \subseteq V$. Let $g_0$ be a holomorphic function in a neighbourhood $U$ of $V$ with $g_0|_V=0$. Then the function $G_0(z^{(1)},z^{(2)},z^{(3)}, z^{(4)})= g_0(z^{(1)},z^{(2)})$ is holomorphic on $U \times \C \times \C$, which is a neighbourhood of $W$ and $G_0|_W=0$. Since $\underline{T}$ is subordinate to $W$, we have that $g_0(T_1,T_2) = G_0(T_1, T_2, 0, 0) = G_0(T_1, T_2, T_3, T_4) = 0$.
		Hence, $(T_1,T_2)$ is subordinate to $V$. By hypothesis,
		$
		\|f(\underline{T})\| = \|f(T_1, T_2, 0, 0)\| = \|h(T_1, T_2)\| \le \|h\|_{\infty,V} = \|f\|_{\infty,W}.
		$
		By Theorem \ref{thm_ext_H}, there exists $\widehat{F} \in H^\infty(\Hex)$ such that $\widehat{F}|_W = f$ and $\|\widehat{F}\|_{\infty,\Hex} = \|f\|_{\infty,W}$. Define $\widehat{h} : \B_2 \to \C$ by $\widehat{h}(z^{(1)},z^{(2)})= \widehat{F}(z^{(1)}, z^{(2)}, 0, 0)$, which is holomorphic on $\B_2$. For $(z^{(1)},z^{(2)}) \in V$, we have that $\widehat{h}(z^{(1)},z^{(2)}) = \widehat{F}(z^{(1)}, z^{(2)}, 0, 0) = f(z^{(1)}, z^{(2)}, 0, 0)= h(z^{(1)},z^{(2)})$ and thus, $\widehat{h}|_V = h$. Finally,
		$\|\widehat{h}\|_{\infty,\B_2} \le \|\widehat{F}\|_{\infty,\Hex} = \|f\|_{\infty,W} = \|h\|_{\infty,V}$ and so, $\|\widehat{h}\|_{\infty,\B_2} = \|h\|_{\infty,V}$.
	\end{proof}		
	
The Toeplitz corona theorem for the Euclidean unit ball
$\B_n=\{z\in\C^n:\|z\|<1\}$ was established in \cite{Amar}. Given
$\varphi_1,\dotsc,\varphi_d\in H^\infty(\B_n)$ and $\epsilon>0$, it
provides necessary and sufficient conditions for the existence of functions
$f_1,\dotsc,f_d\in H^\infty(\B_n)$ such that
\[
\varphi_1f_1+\cdots+\varphi_df_d=1
\quad 
\text{and}
\quad 
\|[f_1,\dotsc,f_d]^t\|_{\infty,\B_n}\le \frac{1}{\epsilon}.
\]
We now present a different version of the Toeplitz corona theorem for $\B_2$ which is associated with a vector-valued analog of the class $\mathfrak{C}(\B_2)$. For Hilbert spaces $\LS_1, \LS_2$, we define 
\[
\mathfrak{C}_{\B_2}(\LS_1, \LS_2)=\{g: \B_2 \to \mathcal{B}(\LS_1, \LS_2): g \ \text{is holomorphic,}  \|g(T_1, T_2)\| \leq 1 \ \text{for all} \ (T_1, T_2) \in \mathfrak{M}_{\B_2} \}.
\]		
Following the arguments as in $(1) \iff (2)$ of Theorem \ref{thm_realization_H_B2} that $g \in \mathfrak{C}_{\B_2}(\LS_1, \LS_2)$ if and only if $g \circ \theta_{\Hex \to \B_2} \in SA_{\Hex}(\LS_1, \LS_2)$. An application of Theorem \ref{thm_TC_H_II} gives a Toeplitz corona type theorem for $\B_2$ as stated below. We leave the proof to the reader.

\begin{thm}\label{thm_TC_B2}
	For $\varphi_1, \dotsc, \varphi_d \in H^\infty(\B_2)$ and $\epsilon>0$, the following statements are equivalent:
	\begin{enumerate}[leftmargin=*]
		\item[$(1)$] there exist $f_1, \dotsc, f_d \in H^\infty(\B_2)$ such that  
		$
		\begin{bmatrix}
			f_1 & \dotsc & f_d
		\end{bmatrix}^t \in \frac{1}{\epsilon}\mathfrak{C}_{\B_2}(\C, \C^d)
		$ 
		and $\overset{d}{\underset{j=1}{\sum}}\varphi_jf_j=1$;
		
		\item[$(2)$] the map
		\[
		(z, w) \mapsto \left(\overset{d}{\underset{j=1}{\sum}} \varphi_j \circ \theta_{\Hex \to \B_2}(z)\overline{\varphi_j\circ \theta_{\Hex \to \B_2}(w)}-\epsilon^2\right)k(z, w)
		\] 
		on $\Hex \times \Hex$ is positive semi-definite for every $k \in AK(\Hex)$;
		
		\item[$(3)$] there exist $\xi \in C(\overline{\D}^2)^+_\Hex, \nabla \in C(\DC)^+_\Hex$ and $\delta \in \C_\Hex^+$ such that for all $z, w \in \Hex$,
		\begin{align*}
			& \overset{d}{\underset{j=1}{\sum}} \varphi_j \circ \theta_{\Hex \to \B_2}(z)\overline{\varphi_j \circ \theta_{\Hex \to \B_2}(w)}-\epsilon^2\\
			& =\xi(z, w)(1-E(z)\overline{E(w)})+\nabla(z, w)(1-e(z)\overline{e(w)})+(1-z^{(3)}\overline{w}^{(3)})\delta(z, w),
		\end{align*}
		where $E(z)$ and $e(z)$ are as in \eqref{eqn_E(z)}.
	\end{enumerate}
\end{thm}

	\section{Realization, interpolation, extension and Toeplitz corona theorems on $\E$}\label{sec_tetra}
	
	\noindent The authors of \cite{Jain} obtained a realization theorem for functions in the Schur-Agler class $SA(\E)$ for the tetrablock, along with an interpolation theorem on $\E$ in terms of admissible kernels on $\E$, and an extension problem on $\E$. Let us recall from Section \ref{sec_hexa} that $\mathfrak{M}_\E$ is the collection of all commuting triples $(T_2, T_3, T_4)$ such that $\|T_3\|<1$ and $\|\Psi(\alpha, T_2, T_3, T_4)\|<1$ for every $\alpha \in \DC$, where $\Psi(\alpha, z^{(2)}, z^{(3)}, z^{(4)})=(\alpha z^{(4)}-z^{(2)})\slash (\al z^{(3)}-1)$. The Schur-Agler class $SA(\E)$ for $\E$ is defined as
	\[
	SA(\E)=\{f \in \text{Hol}(\E): \|f(T_2, T_3, T_4)\|<1 \ \text{for all} \ (T_2, T_3, T_4) \in \mathfrak{M}_\E \}.
	\]
In this section, we also provide a realization theorem for functions in $SA(\E)$, which is presented in terms of functions belonging to $SA(\Hex)$ and admissible kernels on $\Hex$. Moreover, we obtain interpolation and extension theorems on $\E$ as consequences of the corresponding results on $\Hex$ established in Section \ref{sec_hexa}. To begin with, note that
	$(z^{(1)}, z^{(2)}, z^{(3)}) \in \E$ if and only if $(0, z^{(1)}, z^{(2)}, z^{(3)}) \in \Hex$. Furthermore, if $g \in \text{Hol}(\E)$, then the map $g \circ \theta_{\Hex \to \E}$ is holomorphic on $\Hex$, where 
	\[
	\theta_{\Hex \to \E}: \Hex \to \E, \ (z^{(1)}, z^{(2)}, z^{(3)}, z^{(4)})\mapsto (z^{(2)}, z^{(3)}, z^{(4)}).
	\]
	Equipped with this observation, we have the following version of realization theorem on $\E$.
	
	\begin{thm}\label{thm_realization_H_E}
		For a function $g : \E \to \C$, the following are equivalent: 
		\begin{enumerate}[leftmargin=*]
			\item[$(1)$] $g \in SA(\E)$;
			\item[$(2)$] $f=g\circ \theta_{\Hex \to \E} \in SA(\Hex)$;
			\item[$(3)$] $(1-f(z)\overline{f(w)})k(z, w) \succcurlyeq 0$ for all $k \in AK(\Hex)$;
			\item[$(4)$] there exist $\xi \in C(\overline{\D}^2)^+_\Hex, \nabla \in C(\DC)^+_\Hex$ and $\delta \in \C_\Hex^+$ such that for all $z, w \in \Hex$,
			\[
			1-f(z)\overline{f(w)}=\xi(z, w)(1-E(z)\overline{E(w)})+\nabla(z, w)(1-e(z)\overline{e(w)})+(1-z^{(3)}\overline{w}^{(3)})\delta(z, w),
			\]
			where $E(z)$ and $e(z)$ are as in \eqref{eqn_E(z)};
			\item[$(5)$] $f \in UC(\Hex)$.
		\end{enumerate}
		
	\end{thm}
	
	\begin{proof}
		The equivalence of $(2)$ with $(3)-(5)$ follows directly from Theorem \ref{thm_realization_H}. We now prove $(1) \implies (2)$. Let $g \in SA(\E)$ and $\underline{T}=(T_1, T_2, T_3, T_4) \in \mathfrak{M}_\Hex$. It follows from the definition of $\mathfrak{M}_\Hex$ that $(T_2, T_3, T_4) \in \mathfrak{M}_\E$ and so, $\|g\circ \theta_{\Hex \to \E}(\underline{T})\|=\|g(T_2, T_3, T_4)\| \leq 1$. Thus, $g\circ \theta_{\Hex \to \E} \in SA(\Hex)$. It remains to prove $(2) \implies (1)$. Let $g\circ \theta_{\Hex \to \E} \in SA(\Hex)$ and take $(T_2, T_3, T_4) \in \mathfrak{M}_\E$. Evidently, $(0, T_2, T_3, T_4) \in \mathfrak{M}_\E$. Thus, $\|g(T_2, T_3, T_4)\|=\|g\circ \theta_{\Hex \to \E}(0, T_2, T_3, T_4)\| \leq 1$ and so, $g \in SA(\E)$.
	\end{proof}
	
	The following interpolation theorem is obtained as an application of Theorems \ref{thm_interpolation_H} and \ref{thm_realization_H_E}.
	
	\begin{thm}\label{thm_interpolation_H_E}
		Let $F=\{z_1, \dotsc, z_n\} \subseteq \E$ and let $\lm_1, \dotsc, \lm_n \in \DC$. Then the following are equivalent:
		\begin{enumerate}[leftmargin=*]
			\item[$(1)$] there exists $g \in SA(\E)$ such that $g(z_j)=\lm_j$ for $1 \leq j \leq n$;
			\item[$(2)$]there exists $f \in SA(\Hex)$ such that $f(0, z_j)=\lm_j$ for $1 \leq j \leq n$;
			\item[$(3)$] $\begin{bmatrix} (1-\lm_i\overline{\lm_j})k((0, z_i), (0, z_j))\end{bmatrix}_{i, j=1}^n \geq 0$ for all $k \in AK(\Hex)$;
			\item[$(4)$] there exist $\xi \in C(\overline{\D}^2)^+_F, \nabla \in C(\DC)^+_F$ and $\delta \in \C_F^+$ such that
			\begin{align*}
				1-\lambda_i\overline{\lambda}_j
				&=\xi((0, z_i), (0, z_j))(1-E(0, z_i)\overline{E(0, z_j)})\\
				& \quad +\nabla((0, z_i), (0, z_j))(1-e(0, z_i)\overline{e(0, z_j)})+(1-z_i^{(2)}\overline{z}_j^{(2)})\delta((0, z_i), (0, z_j))
			\end{align*}
			for $z_i=(z_i^{(1)}, z_i^{(2)}, z_i^{(3)})$ and $1 \leq i, j \leq n$.
		\end{enumerate}
		In the case when $(1)$ holds, one may choose $f=g\circ \theta_{\Hex \to \E}$ in $(2)$.
	\end{thm}
	
	\begin{proof}
		The part $(1) \implies (2)$ follows from Theorem \ref{thm_realization_H_E} by choosing $f=g\circ \theta_{\Hex \to \E}$. The equivalence of $(2)$ with $(3)-(4)$ follows from Theorem \ref{thm_interpolation_H}. It remains to show that $(2) \implies (1)$. To see this, let $f \in SA(\Hex)$ such that $f(0, z_j)=\lm_j$ for $1 \leq j \leq n$. Define $g: \E \to \C$ as $g(z^{(2)}, z^{(3)}, z^{(4)})=f(0, z^{(2)}, z^{(3)}, z^{(4)})$. For $(T_2, T_3, T_4) \in \mathfrak{M}_\E$, note that  $(0, T_2, T_3, T_4) \in \mathfrak{M}_\Hex$ and so, $\|g(T_2, T_3, T_4)\|=\|f(0, T_2, T_3, T_4)\| \leq 1$. Thus, $g \in SA(\E)$ and $g(z_j)=\lambda_j$ for $ 1\leq j \leq n$.  
	\end{proof}
	
	We now turn our attention to the extension problem on $\E$. As mentioned earlier, the domain $\E$ can embedded inside $\Hex$ as the subset $\{0\} \times \E$. The following example shows that $\{0\} \times \E$ admits the extension property in $SA(\Hex)$.	
	
	\begin{eg}\label{eg_001}
		Consider the subset $W=\{0\} \times \E$ of $\Hex$. Let $f \in \mathscr{HE}(W)$ be non-zero. Consider the maps $g: \Hex \to \C$ and $h: \E \to \C$ given by
		\[
		g(w^{(1)}, w^{(2)}, w^{(3)}, w^{(4)})=f(0, w^{(2)}, w^{(3)}, w^{(4)}) \quad \text{and} \quad h(w^{(2)}, w^{(3)}, w^{(4)})=f(0, w^{(2)}, w^{(3)}, w^{(4)}).
		\]
		Since $f \in \mathscr{HE}(W)$, it extends to a holomorphic function $F$ on a neighbourhood of $W$. Thus, $g(w^{(1)},w^{(2)},w^{(3)}, w^{(4)})=F(0,w^{(2)},w^{(3)}, w^{(4)})$ and $h(w^{(2)},w^{(3)}, w^{(4)})=F(0,w^{(2)},w^{(3)}, w^{(4)})$ are compositions of $F$ with holomorphic coordinate maps, and hence are holomorphic. Also, $g|_W=f$ and $\|g\|_{\infty, \Hex}=\|f\|_{\infty, W}$. Let $\underline{T}=(T_1, T_2, T_3, T_4) \in \mathfrak{M}_\Hex$. Then $(T_2, T_3, T_4) \in \mathfrak{M}_{\E}$. For $r \in (0, 1)$, set $h_r(w^{(2)}, w^{(3)}, w^{(4)})=h(rw^{(2)}, rw^{(3)}, r^2w^{(4)})$ for $(w^{(2)}, w^{(3)}, w^{(4)})\in \overline{\E}$. Clearly, $h_r \in \text{Hol}(\overline{\E})$. Then
		\begin{align*}
			\|g(rT_1, rT_2, rT_3, r^2T_3)\|=\|h(rT_2, rT_3, r^2T_3)\| =\|h_r(T_2, T_3, T_4)\|\leq \|h_r\|_{\infty, \overline{\E}} \leq \|h\|_{\infty, \E} \leq \|f\|_{\infty, W}
		\end{align*}
		for every $r \in (0, 1)$. Letting $r\to 1$, we have that $\|g(\underline{T})\| \leq \|f\|_{\infty, W}$ and thus, $\frac{1}{\|f\|_{\infty, W}}g \in SA(\Hex)$. Hence, $W$ admits the extension property in $SA(\Hex)$. \qed 
	\end{eg}	
	
	Let $Q\E$ be the class of commuting operators $(T_2, T_3, T_4)$ with $\sigma_T(T_2, T_3, T_4) \subseteq \E$ such that
	\[
	\|T_3\|\leq 1 \quad \text{and} \quad \|\Psi(\al, T_2, T_3, T_4)\| \leq 1
	\]
	for all $\al \in \DC$. Evidently, $(T_2, T_3, T_4) \in Q\E$ if and only if $(0, T_2, T_3, T_4)$ is in $Q\Hex$. A commuting triple $(T_2, T_3, T_4) \in Q\E$ is said to be \textit{subordinate} to $V$ if $\sigma_T(T_2, T_3, T_3) \subset V$, and $g(T_2, T_3, T_4)=0$ whenever $g$ is holomorphic in a neighbourhood of $V$ and $g|_V=0$. If $f$ is a function on $V$ that admits a holomorphic extension in a neighbourhood of $V$ and $(T_2, T_3, T_4)$ is subordinate to $V$, then we set $f(T_2, T_3, T_4)=g(T_2, T_3, T_4)$, where $g$ is any holomorphic extension of $f$ in a neighbourhood of $W$. The authors of \cite{Jain} established the following extension theorem on $\E$. Below, we provide an alternative proof here based on Theorem \ref{thm_ext_H}.	
	
	\begin{thm}\label{thm_ext_H_E}
		Let $V \subseteq \E$ and let $h \in \mathscr{HE}(V)$. Then there is a bounded holomorphic function $\widetilde{h}$ on $\E$ such that $\widetilde{h}|_V=h$ and $\|h\|_{\infty, V}=\|\widetilde{h}\|_{\infty, \E}$ if and only if $\|h(A, S, P)\| \leq \|h\|_{\infty, V}$ for every $(A, S, P) \in Q\E$ subordinate to $V$.
	\end{thm}

	\begin{proof} 
		Set $W = \{0\} \times V$ and $W_0=\{0\} \times \E$. Clearly, $W \subseteq W_0 \subseteq \Hex$. Define $f: W \to \C$ as $f(0, z^{(2)}, z^{(3)})=h(z^{(2)}, z^{(3)})$ for all $(z^{(2)}, z^{(3)}) \in V$.  Since $h \in \mathscr{HE}(V)$, it extends to a holomorphic map $h_0$ on a neighbourhood $U$ of $V$. Then the map $F(z^{(1)},z^{(2)},z^{(3)}, z^{(4)}) = h_0(z^{(2)},z^{(3)}, z^{(4)})$ is a holomorphic extension of $f$ to $\mathbb{C} \times U$ such that $F|_W=f$ and so, $f \in \mathscr{HE}(W)$. 
		
		\smallskip 
		
		We first prove the forward part. Suppose there exists a bounded holomorphic function $\widetilde{h}$ on $\E$ such that $\widetilde{h}|_V=h$ and $\|h\|_{\infty, V}=\|\widetilde{h}\|_{\infty, \E}$. The map $\widetilde{f}: W_0 \to \C$ given by $\widetilde{f}(0, z^{(2)}, z^{(3)}, z^{(4)})=\widetilde{h}(z^{(2)}, z^{(3)},z^{(4)})$ has a holomorphic extension on $\Hex$, because $\widetilde{F}(z^{(1)}, z^{(2)}, z^{(3)}, z^{(4)})= \widetilde{h}(z^{(2)}, z^{(3)}, z^{(4)})$ is holomorphic on $\Hex$ that satisfies $\widetilde{F}|_{W_0}=\widetilde{f}$. Thus, $\widetilde{f} \in \mathscr{HE}(W_0)$. By Example \ref{eg_001}, $W_0=\{0\} \times \E$ has the extension property in $SA(\Hex)$. Then 
		\begin{equation}\label{eqn_4004}
			\|\widetilde{f}(\underline{T})\| \leq \|\widetilde{f}\|_{\infty, W_0}\leq \|\widetilde{h}\|_{\infty, \E}=\|h\|_{\infty, V}  
		\end{equation}
		for every $\underline{T} \in Q\Hex$ subordinate to $W_0$. Suppose $(A, S,P) \in Q\E$ is subordinate to $V$. Note that $\sigma_T(0, A, S, P) = \{0\} \times \sigma_T(A, S, P) \subseteq \{0\} \times V = W \subseteq \Hex$ and $(0, A, S, P) \in Q\Hex$. Let $G$ be holomorphic on an open neighbourhood $U \subseteq \mathbb{C}^4$ of $W$ such that $G|_W = 0$. Define $\xi : \mathbb{C}^3 \to \mathbb{C}^4$ as $\xi(z^{(2)}, z^{(3)}, z^{(4)}) = (0, z^{(2)}, z^{(3)}, z^{(4)})$, and $g(z^{(2)}, z^{(3)}, z^{(4)}) = G(0, z^{(2)}, z^{(3)}, z^{(4)})$ for $(z^{(2)}, z^{(3)}, z^{(4)})$ in $\xi^{-1}(U)$. Clearly, $g$ is holomorphic in a neighbourhood of $V$ and $g|_V=0$. Since $(A, S,P)$ is subordinate to $V$ and $g|_V=0$, it follows that
		$g(A, S, P) = 0$. By holomorphic functional calculus, $G(0,A, S,P) = g(A, S,P) = 0$.
		Therefore, $(0,A, S,P)$ is subordinate to $W$. Since $W \subseteq W_0$, $(0,A, S,P)$ is subordinate to $W_0$.  By \eqref{eqn_4004}, 
		\[
		\|h(A, S, P)\|=\|\widetilde{h}(A, S, P)\|=\|\widetilde{F}(0, A, S, P)\|=\|\widetilde{f}(0, A, S, P)\| \leq \|h\|_{\infty, V}
		\]
		for every $(A, S,P)\in Q\E$ subordinate to $V$.
		
		\smallskip 
		
		Conversely, let $\|h(T_2, T_3, T_4)\| \le \|h\|_{\infty,V}$ for every $(T_2, T_3, T_4) \in Q\E$ subordinate to $V$. Let $\underline{T} = (T_1,T_2,T_3, T_4) \in Q\Hex$ be subordinate to $W$.  Clearly, $(T_2, T_3, T_4) \in Q\E$. Moreover, $\sigma_T(\underline{T}) \subseteq W = \{0\} \times V$.
		Let $G(z^{(1)},z^{(2)},z^{(3)}, z^{(4)}) = z^{(1)}$. Clearly, $G$ is holomorphic on $\C^4$ and $G|_W = 0$. Since $\underline{T}$ is subordinate to $W$, it follows that $T_1 = G(T_1,T_2,T_3, T_4) = 0$ and so, $\underline{T} = (0,T_2,T_3, T_4)$. Next, we show that $(T_2, T_3, T_4)$ is subordinate to $V$. By spectral mapping principle, $\sigma_T(T_2, T_3, T_4) \subseteq V$. Suppose $g_0$ is holomorphic in a neighbourhood $U$ of $V$ with $g_0|_V=0$. Then the function $G_0(z^{(1)},z^{(2)},z^{(3)}, z^{(4)})= g_0(z^{(2)},z^{(3)}, z^{(4)})$ is holomorphic on $\mathbb{C} \times U$, which is a neighbourhood of $W$ and $G_0|_W=0$. Since $\underline{T}$ is subordinate to $W$, we have that $g_0(T_2,T_3, T_4) = G_0(0,T_2,T_3, T_4) = G_0(T_1,T_2,T_3, T_4) = 0$.
		Hence, $(T_2,T_3, T_4)$ is subordinate to $V$. By hypothesis,
		\[
		\|f(\underline{T})\| = \|f(0,T_2,T_3, T_4)\| = \|h(T_2, T_3, T_4)\| \leq \|h\|_{\infty,V} = \|f\|_{\infty,W}.
		\]
		By Theorem \ref{thm_ext_H}, there exists $\widehat{F} \in H^\infty(\Hex)$ such that $\widehat{F}|_W = f$ and $\|\widehat{F}\|_{\infty,\Hex} = \|f\|_{\infty,W}$. Define $\widehat{h} : \E \to \C$ by $\widehat{h}(z^{(2)},z^{(3)}, z^{(4)})= \widehat{F}(0,z^{(2)},z^{(3)},z^{(4)})$, which is holomorphic on $\E$. For $(z^{(2)},z^{(3)}, z^{(4)}) \in V$, we have that $\widehat{h}(z^{(2)},z^{(3)}, z^{(4)}) = \widehat{F}(0,z^{(2)},z^{(3)}, z^{(4)}) = f(0,z^{(2)},z^{(3)}, z^{(4)})= h(z^{(2)},z^{(3)}, z^{(4)})$ and thus, $\widehat{h}|_V = h$. Furthermore,
		$\|\widehat{h}\|_{\infty,\E} \le \|\widehat{F}\|_{\infty,\Hex} = \|f\|_{\infty,W} = \|h\|_{\infty,V}$ and so, $\|\widehat{h}\|_{\infty,\E} = \|h\|_{\infty,V}$.
	\end{proof}	
	
For given $\varphi_1,\dotsc,\varphi_d\in H^\infty(\E)$, the authors of \cite{Jain} presented necessary and sufficient conditions for the existence of $f_1, \dotsc, f_d\in H^\infty(\E)$ such that $\varphi_1f_1+\dotsc +\varphi_df_d=1$ with $f_1, \dotsc, f_d$ subjected to a certain norm bound condition. This is referred to as the Toeplitz corona theorem for $\E$. To do so, they introduced the vector-valued Schur-Agler class $SA_\E(\LS_1,\LS_2)$ that consists of all holomorphic functions
$
g:\E\to\mathcal{B}(\LS_1,\LS_2)
$
such that $\|g(T_2,T_3,T_4)\|\leq 1$ for every $(T_2,T_3,T_4)\in\mathfrak{M}_\E$. One verifies by arguments similar to those in the proof of $(1)\iff(2)$ of Theorem \ref{thm_realization_H_E} that a function $g\in SA_\E(\LS_1,\LS_2)$ if and only if 
$g\circ\theta_{\Hex\to\E}\in SA_\Hex(\LS_1,\LS_2)$. Hence, Theorem \ref{thm_TC_H_II} gives the following equivalent formulation of the Toeplitz corona theorem for $\E$, where the characterizations are formulated in terms of admissible kernels on $\Hex$. The details are left to the reader.
	
\begin{thm}\label{thm_TC_E}
	For $\varphi_1, \dotsc, \varphi_d \in H^\infty(\E)$ and $\epsilon>0$, the following statements are equivalent:
	\begin{enumerate}[leftmargin=*]
		\item[$(1)$] there exist $f_1, \dotsc, f_d \in H^\infty(\E)$ such that  
		$
		\begin{bmatrix}
			f_1 & \dotsc & f_d
		\end{bmatrix}^t \in \frac{1}{\epsilon}SA_\E(\C, \C^d)
		$ 
		and $\overset{d}{\underset{j=1}{\sum}}\varphi_jf_j=1$;
		
		\item[$(2)$] the map
		\[
		(z, w) \mapsto \left(\overset{d}{\underset{j=1}{\sum}} \varphi_j \circ \theta_{\Hex \to \E}(z)\overline{\varphi_j\circ \theta_{\Hex \to \E}(w)}-\epsilon^2\right)k(z, w)
		\] 
		on $\Hex \times \Hex$ is positive semi-definite for every $k \in AK(\Hex)$;
		
		\item[$(3)$] there exist $\xi \in C(\overline{\D}^2)^+_\Hex, \nabla \in C(\DC)^+_\Hex$ and $\delta \in \C_\Hex^+$ such that for all $z, w \in \Hex$,
		\begin{align*}
		& \overset{d}{\underset{j=1}{\sum}} \varphi_j \circ \theta_{\Hex \to \E}(z)\overline{\varphi_j \circ \theta_{\Hex \to \E}(w)}-\epsilon^2\\
		& =\xi(z, w)(1-E(z)\overline{E(w)})+\nabla(z, w)(1-e(z)\overline{e(w)})+(1-z^{(3)}\overline{w}^{(3)})\delta(z, w),
		\end{align*}
		where $E(z)$ and $e(z)$ are as in \eqref{eqn_E(z)}.
	\end{enumerate}
\end{thm}		
	
\noindent \textbf{Concluding remark.} The present work identifies the hexablock as a unifying domain for realization, interpolation, extension and Toeplitz corona theorems. In particular, the Schur-Agler class introduced here not only establishes these results for the hexablock itself but also leads to analogous results for the Euclidean unit ball in $\C^2$ and recovers the corresponding theory for the tetrablock. We expect that this viewpoint will be useful in the investigation of function-theoretic and operator-theoretic questions on other domains associated with structured singular values.
	
	\vspace{0.3cm}

	\noindent \textbf{Funding.} The first named author is supported in part by the “Core Research Grant” with Award No. CRG/2023/005223 from Anusandhan National Research Foundation (ANRF) of Govt. of India. During the course of this work, the second named author is supported via the IIT Bombay RDF Grant of the first named author with Project Code RI/0115-10001427. At present, the second named author is supported by a Postdoctoral Fellowship (Ref. No. 0204/9(10)/2026-R\&D-II/4805) awarded by the National Board for Higher Mathematics (NBHM), Government of India.

\end{document}